\documentclass[10pt,reqno]{amsart}
\usepackage{amsmath,amsthm,amssymb}
\usepackage[margin=1in]{geometry}
\usepackage{hyperref}
\usepackage{appendix}
\hypersetup{
    colorlinks=true,
    linkcolor=blue,
    filecolor=magenta,      
    urlcolor=cyan,
    pdfpagemode=FullScreen,
}
    
\usepackage{xcolor}

\newtheorem{theorem}{Theorem}[section]
\newtheorem{lemma}[theorem]{Lemma}
\newtheorem{proposition}[theorem]{Proposition}
\newtheorem{corollary}[theorem]{Corollary}

\theoremstyle{definition}

\newtheorem{definition}[theorem]{Definition}

\newtheorem{remark}[theorem]{Remark}

\newtheorem{example}[theorem]{Example}

\numberwithin{equation}{section}

\newcommand{\mg}[1]{\langle #1 \rangle}

\begin{document}

\title{Left-invariant Hermitian connections on Lie groups with almost Hermitian structures}

\author{David N. Pham}
\author{Fei Ye}
\address{Department of Mathematics $\&$ Computer Science, QCC CUNY, Bayside, NY 11364}
\curraddr{}
\email{}



\dedicatory{}

\begin{abstract}
Left-invariant Hermitian and Gauduchon connections are studied on an arbitrary Lie group $G$ equipped with an arbitrary left-invariant almost Hermitian structure $(\mg{\cdot,\cdot},J)$.  The space of left-invariant Hermitian connections is shown to be in one-to-one correspondence with the space $\wedge^{(1,1)}\mathfrak{g}^\ast\otimes \mathfrak{g}$ of left-invariant 2-forms of type (1,1) (with respect to $J$) with values in $\mathfrak{g}:=\mbox{Lie}(G)$.  Explicit formulas are obtained for the torsion components of every Hermitian and Gauduchon connection with respect to a convenient choice of left-invariant frame on $G$.  The curvature of Gauduchon connections is studied for the special case $G=H\times A$, where $H$ is an arbitrary $n$-dimensional Lie group, $A$ is an arbitrary $n$-dimensional abelian Lie group, and the almost complex structure is totally real with respect to $\mathfrak{h}:=\mbox{Lie}(H)$.  When $H$ is compact, it is shown that $H\times A$ admits a left-invariant (strictly) almost Hermitian structure $(\mg{\cdot,\cdot},J)$ such that the Gauduchon connection corresponding to the Strominger (or Bismut) connection in the integrable case is precisely the trivial left-invariant connection and, in addition, has totally skew-symmetric torsion.  The almost Hermitian structure $(\mg{\cdot,\cdot},J)$ on $H\times A$ is shown to satisfy the \textit{strong K\"{a}hler with torsion} condition.  Furthermore, the affine line of Gauduchon connections on $H\times A$ with the aforementioned almost Hermitian structure is also shown to contain a (nontrivial) flat connection.
\end{abstract}

\date{}

\maketitle

\section{Introduction}

Let $(M,g,J)$ be an almost Hermitian manifold.  A Hermitian connection on $(M,g,J)$ is a connection $\nabla$ on $M$ satisfying $\nabla g=0$ and $\nabla J=0$. The space of Hermitian connections on $(M,g,J)$ is large; it is in one-to-one correspondence with the space $\Omega^{(1,1)}(M;TM)$ (the space of real 2-forms of type (1,1) with respect to $J$ with values in $TM$).  This served as motivation for Gauduchon to introduce a distinguished class of Hermitian connections in \cite{Gauduchon1997}. This distinguished class of Hermitian connections (now called \textit{Gauduchon connections}) is parameterized by $t\in \mathbb{R}$ and forms an affine line of Hermitian connections. When $J$ is integrable, this affine line of Hermitian connections contains both the Chern connection \cite{Chern1979} and the Strominger (or Bismut) connection \cite{Strominger1986,Bismut1989}.  As metric compatible connections, both the Chern and Strominger/Bismut connections are uniquely defined by their torsion tensors.  The Chern connection can be defined as the Hermitian connection whose torsion $T\in \Omega^2(M;TM)$ has vanishing (1,1)-part: $T^{(1,1)}=0$; the Strominger/Bismut connection is characterized by the condition that its torsion $T$ is totally skew-symmetric, that is, $g(T(X,Y),Z)$ is skew-symmetric in $X$, $Y$, and $Z$.  It is interesting to note that the torsion condition $T^{(1,1)}=0$ can always be satisfied (regardless of whether $J$ is integrable or not) whereas the existence of a Hermitian connection with totally skew-symmetric torsion does not hold in general for the non-integrable case (see Appendix \ref{appendix:Bismut} for details).  Historically, the Strominger/Bismut connection was introduced first by Strominger in \cite{Strominger1986} in the context of string theory and later, independently, by Bismut in \cite{Bismut1989} for purely geometric reasons.

It appears that Hermitian connections and, in particular, the smaller class of Gauduchon connections, have been studied primarily for Hermitian manifolds, that is, in the integrable case (see e.g. \cite{YZ2018,Vezzoni2019,Lafuente2022,AV2022} and the references therein).   Motivated by this fact, the current paper studies Hermitian and Gauduchon connections on (strictly) almost Hermitian manifolds. As Lie groups are considerably simpler than general smooth manifolds (since they are inherently algebraic objects by nature), Lie groups serve as an ideal setting for exploring and testing new ideas in differential geometry.  For this reason, we focus our study of Hermitian and Gauduchon connections on Lie groups equipped with left-invariant almost Hermitian structures. 

For a Lie group $G$ equipped with a left-invariant almost Hermitian structure $(\mg{\cdot,\cdot},J)$, the space of left-invariant Hermitian connections is shown to be in one-to-one correspondence with the space $\wedge^{(1,1)}\mathfrak{g}^\ast\otimes \mathfrak{g}$ of left-invariant 2-forms of type (1,1) (with respect to $J$) with values in $\mathfrak{g}:=\mbox{Lie}(G)$ (see Corollary \ref{cor:leftinvariantHermitianOneToOne}). We obtain explicit formulas for the torsion components of every Hermitian and Gauduchon connection with respect to a convenient choice of left-invariant frame on $G$ (which we call a \textit{standard frame}).  The curvature of a Gauduchon connection is studied for the special case $G=H\times A$, where $H$ is an arbitrary $n$-dimensional Lie group, $A$ is an arbitrary $n$-dimensional abelian Lie group, and the almost complex structure $J$ satisfies $J\mathfrak{h}=\mathfrak{a}$ (where $\mathfrak{h}:=\mbox{Lie}(H)$ and $\mathfrak{a}:=\mbox{Lie}(A)$).  Following the terminology of \cite{CPO2010,CPO2011}, $J$ is said to be (almost) totally real with respect to $\mathfrak{h}$.  When $H$ is compact, it is shown that $H\times A$ can be equipped with a left-invariant (strictly) almost Hermitian structure such that the Gauduchon connection corresponding to the Strominger (or Bismut) connection in the integrable case is precisely the trivial left-invariant connection and, in addition, has totally skew-symmetric torsion.   The almost Hermitian structure $(\mg{\cdot,\cdot},J)$ on $H\times A$ is shown to satisfy the \textit{strong K\"{a}hler with torsion} condition (cf \cite{Fino2009}). Furthermore, we show that the affine line of Gauduchon connections on $H\times A$ with the aforementioned almost Hermitian structure also contains a (nontrivial) flat connection.

The rest of the paper is organized as follows.  In Section 2, we give a self-contained review of Hermitian connections where the general formula for the torsion of a Hermitian connection is derived. In addition, notation and conventions are established which will be used in the remainder of the paper. The free parameter in the torsion formula is shown to be an element $\alpha\in \Omega^{(1,1)}(M;TM)$ which, in turn, establishes a one-to-one correspondence between the space of Hermitian connections and the space $\Omega^{(1,1)}(M;TM)$.  We note that the formula obtained in Section 2 differs in appearance from that obtained originally by Gauduchon in \cite{Gauduchon1997}, but is shown to be equivalent.  From the general torsion formula, we obtain the torsion formula for the Gauducuhon connections by an appropriate choice of $\alpha^t\in \Omega^{(1,1)}(M;TM)$ which is parameterized by $t\in \mathbb{R}$. In Section 3, we turn our attention to left-invariant Hermitian and Gauduchon connections on arbitrary Lie groups equipped with left-invariant almost Hermitian structures. Explicit formulas are obtained for the components of the torsion tensors relative to a so-called ``standard frame".  (As some of the formulas are on the lengthy side, we also verified the formulas numerically using the program Octave to ensure that the formulas worked as expected.)  Lastly, in Section 4, we study the curvature of Lie groups of the form $H\times A$ equipped with left-invariant almost Hermitian structures where the almost complex structure is totally real with respect to $\mathfrak{h}$.

\section{Preliminaries}
\label{sec:Preliminaries}
\noindent In this section, we review the relevant background and establish notation and conventions that we will use for the rest of the paper. Let $(M,g,J,\omega)$ be an almost Hermitian manifold where 
\begin{itemize}
    \item $g$ is the Hermitian metric
    \item $J$ is the almost complex structure
    \item $\omega$ is the fundamental 2-form defined by $\omega(\cdot,\cdot):=g(J\cdot, \cdot)$
\end{itemize}
For a vector field $X$ (real or complex) on $M$, we let $X=X^++X^-$ denote its decomposition into its (1,0) and (0,1) components.  Explicitly,
\begin{equation}
    X^+:=\frac{1}{2}(X-iJX),\hspace*{0.2in}X^-:=\frac{1}{2}(X+iJX)
\end{equation}
and one easily verifies that $JX^+=iX^+$ and $JX^-=-iX^-$.  We adopt the following convention for the Nijenhuis tensor:
\begin{equation}
    \label{def:Nijenhuis}
    N(X,Y):=J[JX,Y]+J[X,JY]+[X,Y]-[JX,JY].
\end{equation}
\subsection{\texorpdfstring{$TM$}{TM}-valued forms of type (1,1)}
\label{subsection(1,1)-forms}
Let $\Omega^{(1,1)}(M;TM)$ to denote the space of real (1,1)-forms with values in $TM$.  For a 3-form $\psi\in \Omega^3(M)$, we follow \cite{Gauduchon1997} and decompose $\psi$ as
$$
\psi = \psi^++\psi^-
$$
where $\psi^+$ denotes its $(2,1)+(1,2)$-part and $\psi^-$ denotes its $(3,0)+(0,3)$-part.  

When $J$ is integrable, we have the following result:
\begin{proposition}
    \label{prop:dOmegaNzero}
    If $J$ is integrable, then $(d\omega)^+=d\omega$ and 
    \begin{equation}
        \label{eq:dOmegaNzeroProp1}
        d\omega(JX,JY,JZ)=d\omega(JX,Y,Z)+d\omega(X,JY,Z)+d\omega(X,Y,JZ).
    \end{equation}
\end{proposition}
\begin{proof}
    Since $J$ is integrable, the Lie bracket of two (1,0)-vector fields is again a (1,0)-vector field and since $\omega$ is (1,1), we have
    $$
        \omega(X^+,Y^+)=\omega([X^+,Y^+],Z^+)=0.
    $$
    The invariant formula for the exterior derivative then implies $d\omega(X^+,Y^+,Z^+)=0$. The same argument also implies that $d\omega(X^-,Y^-,Z^-)=0$.  Of course, the latter can also be seen by conjugating $d\omega(X^+,Y^+,Z^+)=0$ (and using the fact that $\omega$ is real).  This proves that $(d\omega)^+=d\omega$ when $J$ is integrable.  Since $d\omega$ has no (3,0) or (0,3) part and (\ref{eq:dOmegaNzeroProp1}) is totally skew-symmetric on both sides of the equation, to verify (\ref{eq:dOmegaNzeroProp1}) it suffices to check that equality holds when $X=X^+$, $Y=Y^+$, and $Z=Z^-$:
    \begin{align*}
         &d\omega(JX^+,Y^+,Z^-)+d\omega(X^+,JY^+,Z^-)+d\omega(X^+,Y^+,JZ^-)\\
         &=id\omega(X^+,Y^+,Z^-)+id\omega(X^+,Y^+,Z^-)-id\omega(X^+,Y^+,Z^-)\\
         &=id\omega(X^+,Y^+,Z^-)\\
         &=(i)(i)(-i)d\omega(X^+,Y^+,Z^-)\\
         &=d\omega(JX^+,JY^+,JZ^-).
    \end{align*}
    Hence, equality holds whenever the arguments consists of two (1,0)-vectors and one (0,1)-vector.  Conjugation now implies that (\ref{eq:dOmegaNzeroProp1}) also holds whenever the arguments consists of two (0,1)-vectors and one (1,0)-vector.
\end{proof}
\noindent More generally, we have the following:
\begin{proposition}
    \label{prop:3form(2,1)+(1,2)}
    For $\eta\in \Omega^3(M)$, 
    \begin{align}
        \label{eq:3form(2,1)+(1,2)}
        \eta^+(X,Y,Z)=\frac{1}{4}\left[3\eta(X,Y,Z)+\eta(JX,JY,Z)+\eta(JX,Y,JZ)+\eta(X,JY,JZ)\right]
    \end{align}
\end{proposition}
\begin{proof}
    This is a straightforward but somewhat tedious calculation.  Expanding 
    \begin{align*}
        \eta^+(X,Y,Z)&=\eta(X^+,Y^+,Z^-)+\eta(X^+,Y^-,Z^+)+\eta(X^-,Y^+,Z^+)\\
        &+\eta(X^-,Y^-,Z^+)+\eta(X^-,Y^+,Z^-)+\eta(X^+,Y^-,Z^-)
    \end{align*}
    and using the fact that $X^{\pm}=\frac{1}{2}(X\mp iJX)$ gives (\ref{eq:3form(2,1)+(1,2)}).
\end{proof}
\noindent Let $\Omega^{3+}(M)$ be the space of real valued 3-forms of type $(2,1)+(1,2)$.  The next result shows that every element of $\Omega^{3+}(M)$ is induced by an element of $\Omega^{(1,1)}(M;TM)$.
\begin{proposition}
    \label{prop:(1,1)to(2,1)+(2,1)}
    Let $F:\Omega^{(1,1)}(M;TM)\rightarrow \Omega^{3+}(M)$, $\alpha\mapsto F(\alpha)$ be the linear map defined by 
    $$
        F(\alpha)(X,Y,Z):=g(\alpha(X,Y),Z)+g(\alpha(Y,Z),X)+g(\alpha(Z,X),Y).
    $$
    Then $F$ is surjective. 
    \end{proposition}
    \begin{proof}
        Since $\alpha$ is skew-symmetric and (1,1), it follows immediately that the $F(\alpha)$ is totally skew-symmetric and of type $(2,1)+(1,2)$. To see that the map is surjective, let $\eta\in \Omega^{3+}(M)$.  Define $\alpha_\eta\in \Omega^{(1,1)}(M;TM)$ via
        $$
            g(\alpha_\eta(X,Y),Z)=\frac{1}{4}[\eta(X,Y,Z)+\eta(JX,JY,Z)].
        $$
        From the definition of $\alpha_\eta$, we have $\alpha_\eta(JX,JY)=\alpha(X,Y)$ which shows that $\alpha_\eta$ is (1,1).  Using Proposition \ref{prop:3form(2,1)+(1,2)}, one has $F(\alpha_\eta)=\eta^+=\eta$.
    \end{proof}
 We define $\Omega_s^{(1,1)}(M;TM):=\mbox{ker}~ F$. Extend the metric $g$ to the vector bundle of real (1,1)-forms with values in $TM$ and define $\Omega^{(1,1)}_a(M;TM)$ to be the orthogonal complement of $\Omega_s^{(1,1)}(M;TM)$ with respect to $g$:
  $$
  \Omega^{(1,1)}(M;TM)=\Omega^{(1,1)}_s(M;TM)\oplus \Omega^{(1,1)}_a(M;TM)
  $$
\begin{corollary}
    \label{cor:(1,1)-aIsomorphism}
    The linear map $F:\Omega^{(1,1)}(M;TM)\rightarrow \Omega^{3+}(M)$ restricted to $\Omega^{(1,1)}_a(M;TM)$ is an isomorphism.
\end{corollary}
\noindent For convenience, we record the following fact:
\begin{proposition}
    \label{prop:Omega(1,1)s}
    Let $\alpha\in \Omega^{(1,1)}(M;TM)$.  Then $\alpha\in \Omega^{(1,1)}_s(M;TM)$ if and only if 
    \begin{equation}
     \label{eq:Omega(1,1)s}
        \omega(\alpha(X,Y),Z)+\omega(\alpha(Y,Z),X)+\omega(\alpha(Z,X),Y)=0.
    \end{equation}
\end{proposition}
\begin{proof}
    Using $-g(\cdot,J\cdot)=\omega(\cdot,\cdot)$ and the fact that $\alpha(J\cdot,J\cdot)=\alpha(\cdot,\cdot)$, we have
    \begin{align*}
        \omega(\alpha(X,Y),Z)&+\omega(\alpha(Y,Z),X)+\omega(\alpha(Z,X),Y)\\
            &=-g(\alpha(JX,JY),JZ)-g(\alpha(JY,JZ),JX)-g(\alpha(JZ,JX),JY)
    \end{align*}
    which implies the proposition.
\end{proof}

\subsection{Metric compatibility and torsion}
\noindent Recall that for a linear connection $\nabla$ on $M$, the torsion of $\nabla$ is defined by
\begin{equation}
    T^\nabla(X,Y):=\nabla_XY-\nabla_YX-[X,Y].
\end{equation}
We recall the following well known fact:
\begin{proposition}
    \label{prop:TorsionUniqueness}
    For any Riemannian manifold $(M,g)$ and any $TM$-valued $2$-form $T\in \Omega^2(M;TM)$, there exists a unique connection $\nabla$ satisfying $\nabla g=0$ and $T^\nabla=T$.  Explicitly, $\nabla$ is defined by
    \begin{equation}
    \begin{aligned}
    2g(\nabla_XY,Z)
        =& Xg(Y,Z)-Zg(X,Y)+Yg(Z,X)\\
        & + g([X,Y],Z)-g([Y,Z],X)-g([X,Z],Y)\\
        & + g(T(X,Y),Z)-g(T(Y,Z),X)-g(T(X,Z),Y)        
    \end{aligned} \label{eq:UniqueConnectionTorsionEq}
    \end{equation}
\end{proposition}

\subsection{Hermitian connections}
In this section, we give a self-contained review of Hermitian connections.  For a more detailed account, we refer the reader to \cite{Gauduchon1997}.
\begin{definition}
    \label{def:HermitianDef}
    A connection $\nabla$ on an almost Hermitian manifold $(M,g,J,\omega)$ is Hermitian if $\nabla g=0$ and $\nabla J=0$.
\end{definition}
\noindent For convenience, we make the following definition:
\begin{definition}
\label{def:thetaJDef}
    For a $TM$-valued 2-form $\theta\in \Omega^2(M;TM)$, $\theta_J\in \Omega^2(M;TM)$ is defined as
    \[
        \theta_J(X,Y):=J\theta(JX,Y)+J\theta(X,JY)+\theta(X,Y)-\theta(JX,JY)
    \]
\end{definition}
\noindent For convenience, we record some identities related to $\theta_J$. The identities follow by straightforward calculation.
\begin{lemma}
\label{lemma:thetaJBasicIdentities}
    Let $\theta\in \Omega^2(M;TM)$. Then
    \begin{enumerate}
        \item ${(\theta_J)}_J=4\theta_J$,
        \item $\theta_J(X^+, Y^-)=0$,
        \item $\frac{1}{2}\theta_J(X^+, Y^+)=\theta(X^+, Y^+)+iJ\theta(X^+, Y^+)$,
        \item $J\theta_J(X^+, Y^+)=-i\theta_J(X^+, Y^+)$.
    \end{enumerate}
\end{lemma}

\begin{lemma}
    \label{lemma:HermitianTorsion}
    If $\nabla$ is a Hermitian connection on $(M,g,J,\omega)$, then its torsion $T$ satisfies
    \begin{itemize}
        \item[(1)] $T_J(X,Y)+N(X,Y)=0$ 
    \item[(2)] $d\omega(X,Y,Z)=\omega(T(X,Y),Z)+\omega(T(Y,Z),X)+\omega(T(Z,X),Y)$.
    \end{itemize}
\end{lemma}
\begin{proof}
It follows immediately from Definition \ref{def:HermitianDef} that a connection $\nabla$ is Hermitian if and only if $\nabla J=0$ and $\nabla \omega=0$.  For any vector fields $X$, $Y$, and $Z$, we have
\begin{equation}
(\nabla_X\omega)(Y, Z)= X(\omega(Y, Z))-\omega(\nabla_XY, Z)-\omega(Y, \nabla_XZ)
\label{eq:nablaomega}
\end{equation}
Equation (\ref{eq:nablaomega}) together with the definition of $T$ implies
\begin{align}
\nonumber
d\omega(X, Y, Z)
=&X\omega(Y, Z)-Y\omega(X, Z)+Z\omega(X, Y)\\
\nonumber
&-\omega([X, Y], Z)+\omega([X, Z], Y)-\omega([Y, Z], X)\\
\nonumber
=&X\omega(Y, Z)-Y\omega(X, Z)+Z\omega(X, Y)\\
\nonumber
& - \omega(\nabla_XY, Z)+\omega(\nabla_YX, Z)+\omega(T(X,Y), Z)\\
\nonumber
& + \omega(\nabla_XZ, Y)-\omega(\nabla_ZX, Y)-\omega(T(X,Z), Y)\\
\nonumber
& - \omega(\nabla_YZ, X)+\omega(\nabla_ZY, X)+\omega(T(Y,Z), X)\\
\nonumber
=&(\nabla_X\omega)(Y, Z) + (\nabla_Y\omega)(Z,X) + (\nabla_Z\omega)(X, Y)\\
\label{eq:domegaTorsion}
& +\omega(T(X,Y), Z) + \omega(T(Y,Z), X) + \omega(T(Z,X), Y)
\end{align}

\noindent For any vector fields $X$, $Y$, and $Z$, we have $\nabla_X(JY)=(\nabla_XJ)Y+J\nabla_XY$.
Since
\begin{align*}
    T(X, Y)+[X, Y]=&\nabla_XY-\nabla_{Y}X\\
    T(JX, JY)+[JX, JY]=&\nabla_{JX}(JY)-\nabla_{JY}(JX)\\
    JT(JX, Y)+J[JX, Y]=&J\nabla_{JX}Y-J\nabla_{Y}(JX)\\
    JT(X, JY)+J[X, JY]=&J\nabla_{X}(JY)-J\nabla_{JY}X
\end{align*}
\noindent we have
\begin{align}
\nonumber
T_J(X, Y)+N(X, Y)&= J\nabla_{JX}Y-J\nabla_{Y}(JX) + J\nabla_{X}(JY)-J\nabla_{JY}X\\
\nonumber
& + \nabla_XY-\nabla_{Y}X - \nabla_{JX}(JY)+\nabla_{JY}(JX)\\
\nonumber
=& -\nabla_{JX}(JY)+J\nabla_{JX}Y + \nabla_{JY}(JX)-J\nabla_{JY}X\\
\nonumber
& + J\nabla_{X}(JY) + \nabla_XY - J\nabla_{Y}(JX)-\nabla_{Y}X\\
\label{eq:JTorsion}
=& -(\nabla_{JX}J)Y + (\nabla_{JY}J)X + J(\nabla_{X}J)Y - J(\nabla_{Y}J)X
\end{align}
It follows immediately from (\ref{eq:domegaTorsion}) and (\ref{eq:JTorsion}) that if $\nabla$ is Hermitian then its torsion satisfies conditions (1) and (2) of Lemma \ref{lemma:HermitianTorsion}.    
\end{proof}

\begin{proposition}
    \label{prop:HermitianTorsion}
    A connection $\nabla$ on $(M,g,J,\omega)$ satisfying $\nabla g=0$ is Hermitian if and only if its torsion $T$ satisfies both conditions of Lemma \ref{lemma:HermitianTorsion}.
\end{proposition}
\begin{proof}
From Lemma \ref{lemma:HermitianTorsion}, it remains to show that conditions (1) and (2) are sufficient. So suppose that $\nabla$ is a $g$-compatible connection which satisfies conditions (1) and (2) of Proposition \ref{prop:HermitianTorsion}.

Using Proposition \ref{prop:TorsionUniqueness}, the fact that $\omega$ is $J$-invariant, and $g(\cdot,\cdot)=\omega(\cdot,J\cdot)$, we have
\begin{align*}
    2g((\nabla_XJ)Y,Z)&=2g(\nabla_X(JY),Z)-2g(J\nabla_XY,Z)\\
            &=2g(\nabla_X(JY),Z)+2g(\nabla_XY,JZ)\\
            =& X\omega(JY,JZ)+Z\omega(X,Y)+(JY)\omega(Z,JX)\\
         & + \omega([X,JY],JZ)-\omega([JY,Z],JX)+\omega([X,Z],Y)\\
         & + \omega(T(X,JY),JZ)-\omega(T(JY,Z),JX)+\omega(T(X,Z),Y) \\
         & -X\omega(Y,Z)-(JZ)\omega(X,JY)+Y\omega(JZ,JX)\\
         & - \omega([X,Y],Z)-\omega([Y,JZ],JX)-\omega([X,JZ],JY)\\
         & - \omega(T(X,Y),Z)-\omega(T(Y,JZ),JX)-\omega(T(X,JZ),JY) \\
         &=d\omega(X,Y,Z)+\omega([Y,Z],X)-d\omega(X,JY,JZ)-\omega([JY,JZ],X)\\
         &-\omega([JY,Z],JX)+ \omega(T(X,JY),JZ)-\omega(T(JY,Z),JX)+\omega(T(X,Z),Y)\\
         &-\omega([Y,JZ],JX)- \omega(T(X,Y),Z)-\omega(T(Y,JZ),JX)-\omega(T(X,JZ),JY)
\end{align*}
The last line can be rewritten as
 \begin{align*}
     2g((\nabla_XJ)Y,Z)&=d\omega(X,Y,Z)+\omega([Y,Z],X)-d\omega(X,JY,JZ)-\omega([JY,JZ],X)\\
           &+\omega(J[JY,Z],X)- \omega(JT(X,JY),Z)+\omega(JT(JY,Z),X)+\omega(T(X,Z),Y)\\
           &+\omega(J[Y,JZ],X)- \omega(T(X,Y),Z)+\omega(JT(Y,JZ),X)+\omega(JT(X,JZ),Y)\\
           &=d\omega(X,Y,Z)-d\omega(X,JY,JZ)+\omega(N(Y,Z),X)+\omega(T_J(Y,Z),X)\\
           &-\omega(T(Y,Z),X)+\omega(T(JY,JZ),X)+ \omega(T(X,JY),JZ)+\omega(T(X,Z),Y)\\
           &- \omega(T(X,Y),Z)-\omega(T(X,JZ),JY)    
 \end{align*}
Applying condition (1) of Proposition \ref{prop:HermitianTorsion} yields
$$
\omega(N(Y,Z),X)+\omega(T_J(Y,Z),X)=0
$$
Applying condition (2) of Proposition \ref{prop:HermitianTorsion} gives
$$
d\omega(X,Y,Z)- \omega(T(X,Y),Z)+\omega(T(X,Z),Y)-\omega(T(Y,Z),X)=0
$$
and
$$
-d\omega(X,JY,JZ)+\omega(T(JY,JZ),X)+ \omega(T(X,JY),JZ)-\omega(T(X,JZ),JY) =0
$$
From this, it follows that $2g((\nabla_XJ)Y,Z)=0$ which shows that $\nabla J=0$.  Hence, $\nabla$ is Hermitian. 
\end{proof}
\begin{remark}
    From Proposition \ref{prop:HermitianTorsion}, we recover the well known fact that the Levi-Civitia connection is Hermitian if and only if $J$ is integrable and $\omega$ is closed.  In other words, the Levi-Civitia connection associated to $g$ is Hermitian precisely when $(M,g,J,\omega)$ is K\"{a}hler.
\end{remark}

\begin{lemma}
    \label{lemma:thetaConditions}
    A connection on $(M,g,J,\omega)$ satisfying $\nabla g=0$ is Hermitian if and only if its torsion is of the form $T=-\frac{1}{4}N+\theta$ where $\theta\in \Omega^2(M;TM)$ satisfies the following conditions:
    \begin{itemize}
        \item[(1)] $\theta(X^+,Y^+)$ is (1,0) for all (1,0) vector fields $X^+,Y^+$
        \item[(2)] $d\omega(X^+,Y^+,Z^-)=\omega(\theta(X^+,Y^+),Z^-)+\omega(\theta(Y^+,Z^-),X^+)+\omega(\theta(Z^-,X^+),Y^+)$ for all (1,0) vector fields $X^+,Y^+$ and all (0,1) vector fields $Z^-$.
    \end{itemize}
\end{lemma}
\begin{proof}
    By Proposition \ref{prop:HermitianTorsion}, a $g$-compatible connection is Hermitian if and only if its torsion satisfies both conditions of Lemma \ref{lemma:HermitianTorsion}.  Let $T\in \Omega^2(M;TM)$ be any $TM$-valued 2-form and let $\theta:=\frac{1}{4}N+T$.  Then
    $$
        \theta_J=\left(\frac{1}{4}N+T\right)_J=\frac{1}{4}N_J+T_J=\frac{1}{4}(4N)+T_J=N+T.
    $$
    Hence, $T$ satisfies condition (1) of Lemma \ref{lemma:HermitianTorsion} if and only if $\theta_J=0$.  So any $T\in \Omega^2(M;TM)$ which satisfies condition (1) of Lemma \ref{lemma:HermitianTorsion} is of the form $T=-\frac{1}{4}N+\theta$ where $\theta\in \Omega^2(M;TM)$ satisfies $\theta_J=0$.   It follows from Lemma \ref{lemma:thetaJBasicIdentities} that $\theta_J=0$ is equivalent to the statement that $\theta(X^+,Y^+)$ is (1,0) for all $(1,0)$-vector fields $X^+,Y^+$.

    Now suppose $T\in \Omega^2(M;TM)$ satisfies both conditions of Lemma \ref{lemma:HermitianTorsion}.  By the above argument, $T$ must take the form $T=-\frac{1}{4}N+\theta$ for some $\theta\in \Omega^2(M;TM)$ which satisfies condition (1) of Lemma \ref{lemma:thetaConditions}.  Expanding condition (2) of Lemma \ref{lemma:HermitianTorsion} gives
    \begin{align}
        \nonumber
        d\omega(X,Y,Z)&=\omega(T(X,Y),Z)+\omega(T(Y,Z),X)+\omega(T(Z,X),Y)\\
        \label{eq:domegaNtheta}
        &=-\frac{1}{4}\left(\omega(N(X,Y),Z)+\mbox{c.p}\right)+\left(\omega(\theta(X,Y),Z)+\mbox{c.p}\right)
    \end{align}
    where ``c.p" denotes cyclic permutation.  Note that $N(X^+,Y^-)=0$, $N(X^+,Y^+)$ is (0,1), and (by conjugation) $N(X^-,Y^-)$ is (1,0).  Since $\omega$ is a (1,1)-form, equation (\ref{eq:domegaNtheta}) reduces to the following when the arguments are $X^+,Y^+,Z^+$:
    \begin{equation}
        \label{eq:domegaIdentity}
        d\omega(X^+,Y^+,Z^+)=-\frac{1}{4}\left(\omega(N(X^+,Y^+),Z^+)+\mbox{c.p}\right).
    \end{equation}
    Now, by expanding both sides of (\ref{eq:domegaIdentity}), one finds that (\ref{eq:domegaIdentity}) is actually an identity.  Hence, when the arguments are $X^+,Y^+,Z^+$ (or $X^-,Y^-,Z^-$), equation (\ref{eq:domegaNtheta}) is always satisfied (provided $\theta$ satisfies condition (1) of Lemma \ref{lemma:thetaConditions}). 

    When one takes skew-symmetry and conjugation into account (and assuming $\theta$ satisfies condition (1) of Lemma \ref{lemma:thetaConditions}), it follows that (\ref{eq:domegaNtheta}) holds in general if and only if it holds for $X^+,Y^+,Z^-$.  Substituting these arguments into (\ref{eq:domegaNtheta}) and simplifying gives
    \begin{equation}
        \label{eq:X+Y+Z+}
        d\omega(X^+,Y^+,Z^-)=\omega(\theta(X^+,Y^+),Z^-)+\omega(\theta(Y^+,Z^-),X^+)+\omega(\theta(Z^-,X^+),Y^+).
    \end{equation}
    From this, we see that condition (2) of Lemma \ref{lemma:HermitianTorsion} holds if and only if $\theta$ satisfies (\ref{eq:X+Y+Z+}) (as well as condition (1) of Lemma \ref{lemma:thetaConditions}).  This completes the proof.
\end{proof}

\begin{corollary}
    \label{cor:thetaFormula}
    A $g$-compatible connection on $(M,g,J,\omega)$ is Hermitian if and only if its torsion is of the form $T=-\frac{1}{4}N+\alpha+\beta$ where $\alpha$ is an arbitrary real $TM$-valued 2-form of type (1,1) and $\beta$ is a real $TM$ valued 2-form of type $(2,0)+(0,2)$ which is uniquely defined by the following conditions:
    \begin{itemize}
        \item[(1)] $\beta(X^+,Y^+)$ is (1,0) for all (1,0)-vector fields $X^+,Y^+$ and
        \item[(2)] $\omega(\beta(X^+,Y^+),Z^-)=d\omega(X^+,Y^+,Z^-)-\omega(\alpha(Y^+,Z^-),X^+)-\omega(\alpha(Z^-,X^+),Y^+)$
    \end{itemize}
\end{corollary}
\begin{proof} 
    From Lemma \ref{lemma:thetaConditions}, a $g$-compatible connection is Hermitian if and only if its torsion is of the form $T=-\frac{1}{4}N+\theta$ where $\theta\in \Omega^2(M;TM)$ satisfies conditions (1) and (2) of Lemma \ref{lemma:thetaConditions}.  Decompose $\theta$ as a sum of its (2,0), (1,1), and (0,2) parts: 
    $$
        \theta=\theta^{(2,0)}+\theta^{(1,1)}+\theta^{(0,2)}.
    $$
    Since $\theta$ is real, we have $\theta^{(0,2)}=\overline{\theta^{(2,0)}}$ and $\overline{\theta^{(1,1)}}=\theta^{(1,1)}$.  Let 
    $$
        \alpha=\theta^{(1,1)},\hspace*{0.2in}\beta=\theta^{(2,0)}+\theta^{(0,2)}.
    $$
    Then conditions (1) and (2) of Lemma \ref{lemma:thetaConditions} are precisely those of Corollary \ref{cor:thetaFormula}. 
\end{proof}

\noindent The next result expresses the torsion formula given by Corollary \ref{cor:thetaFormula} in terms of real vector fields.
\begin{proposition}
    \label{prop:HermitianTorsionFormulaReal}
    A $g$-compatible connection on $(M,g,J,\omega)$ is Hermitian if and only if its torsion $T$ is of the form
    \begin{align*}
        g(T(X,Y),Z)&=-\frac{1}{4}g(N(X,Y),Z)-\frac{1}{2}(d\omega)^+(JX,JY,JZ)+\frac{1}{2}(d\omega)^+(X,Y,JZ)\\
        &+\frac{1}{2}\alpha^+(X,Y,Z)-\frac{1}{2}\alpha^+(JX,JY,Z)+g(\alpha(X,Y),Z)
    \end{align*}
    where $\alpha\in \Omega^{(1,1)}(M;TM)$ is a real arbitrary $TM$-valued (1,1)-form and 
    $$
        \alpha^+(X,Y,Z):=g(\alpha(X,Y),Z)+\mbox{cyclic}.
    $$
\end{proposition}
\begin{proof}
    Let $T=-\frac{1}{4}N+\alpha+\beta$ as in Corollary \ref{cor:thetaFormula}.  Expanding $\omega(\beta(X,Y),Z)$ gives
    \begin{align*}
        \omega(\beta(X,Y),Z)&=\omega(\beta(X^+,Y^+),Z^-)+\omega(\beta(X^-,Y^-),Z^+)\\
        &=d\omega(X^+,Y^+,Z^-)-\omega(\alpha(Y^+,Z^-),X^+)-\omega(\alpha(Z^-,X^+),Y^+)\\
        &+d\omega(X^-,Y^-,Z^+)-\omega(\alpha(Y^-,Z^+),X^-)-\omega(\alpha(Z^+,X^-),Y^-)  
    \end{align*}
    Let
    $$
        \eta^+(X,Y,Z)=\eta(X,Y,Z):=\omega(\alpha(X,Y),Z)+\omega(\alpha(Y,Z),X)+\omega(\alpha(Z,X),Y)
    $$
     Since $\omega(\alpha(X^+,Y^+),Z)=\omega(\alpha(X^-,Y^-),Z)=0$, the above can be rewritten as
    \begin{align*}
        \omega(\beta(X,Y),Z)&=d\omega(X^+,Y^+,Z^-)+d\omega(X^-,Y^-,Z^+)-\eta^+(X^+,Y^+,Z^-)-\eta^+(X^-,Y^-,Z^+)  
    \end{align*}
    Using the fact that $X^{\pm}=\frac{1}{2}(X\mp iJX)$ gives 
    \begin{align*}
        \omega(\beta(X,Y),Z)&=\frac{1}{4}\left[d\omega(X,Y,Z)-d\omega(JX,JY,Z)+d\omega(JX,Y,JZ)+d\omega(X,JY,JZ)\right]\\
        &-\frac{1}{4}\left[\eta^+(X,Y,Z)-\eta^+(JX,JY,Z)+\eta^+(JX,Y,JZ)+\eta^+(X,JY,JZ)\right]
    \end{align*}
    Rewriting the left side as $g(\beta(X,Y),Z)$ gives
     \begin{align*}
        g(\beta(X,Y),Z)&=\frac{1}{4}\left[d\omega(X,Y,JZ)-d\omega(JX,JY,JZ)-d\omega(JX,Y,Z)-d\omega(X,JY,Z)\right]\\
        &-\frac{1}{4}\left[\eta^+(X,Y,JZ)-\eta^+(JX,JY,JZ)-\eta^+(JX,Y,Z)-\eta^+(X,JY,Z)\right]
    \end{align*}
    Using Proposition \ref{prop:3form(2,1)+(1,2)} and the fact that $\eta=\eta^+$, the above can be rewritten as 
    \begin{align*}
        g(\beta(X,Y),Z)&=-\frac{1}{2}(d\omega)^+(JX,JY,JZ)+\frac{1}{2}(d\omega)^+(X,Y,JZ)+\frac{1}{2}\eta^+(JX,JY,JZ)-\frac{1}{2}\eta^+(X,Y,JZ)
    \end{align*}
    Defining
    $$
        \alpha^+(X,Y,Z):=g(\alpha(X,Y),Z)+\mbox{cyclic}
    $$
    gives
    \begin{align*}
        g(\beta(X,Y),Z)&=-\frac{1}{2}(d\omega)^+(JX,JY,JZ)+\frac{1}{2}(d\omega)^+(X,Y,JZ)+\frac{1}{2}\alpha^+(X,Y,Z)-\frac{1}{2}\alpha^+(JX,JY,Z)
    \end{align*}
    Substituting the above expression into 
        $$
            g(T(X,Y),Z)=-\frac{1}{4}g(N(X,Y),Z)+g(\beta(X,Y),Z)+g(\alpha(X,Y),Z)
        $$
    gives the desired form.
\end{proof}

\begin{remark}
    The general torsion formula for Hermitian connections given in \cite{Gauduchon1997} consists of two free quantities: a real $TM$-valued (1,1)-form $B_s\in \Omega^{(1,1)}_s(M;TM)$ and a real valued $3$-form $\psi^+$ of type (2,1)+(1,2) where we recall that $B_s$ satisfies the cyclic condition
    $$
        g(B_s(X,Y),Z)+\mbox{cyclic}=0
    $$
    From Corollary \ref{cor:(1,1)-aIsomorphism} the space of real valued $3$-forms of type $(2,1)+(1,2)$ is isomorphic to the space $\Omega^{(1,1)}_a(M;TM)$.  Since $\Omega^{(1,1)}(M;TM)=\Omega^{(1,1)}_s(M;TM)\oplus \Omega^{(1,1)}_a(M;TM)$, we see that the two free quantities in the formula of \cite{Gauduchon1997} is equivalent to a $TM$-valued (1,1)-form as given by Proposition \ref{prop:HermitianTorsionFormulaReal}.
     Given $B_s\in \Omega^{(1,1)}_s(M;TM)$ and $\psi^+$, we can recover Gauduchon's exact formula by defining $\alpha\in \Omega^{(1,1)}(M;TM)$ in Proposition \ref{prop:HermitianTorsionFormulaReal} by
      \begin{align*}
          g(\alpha(X,Y),Z)&=\frac{1}{8}\left[(d\omega)^+(JX,JY,JZ)+(d\omega)^+(X,Y,JZ)\right]\\
          &+\frac{3}{8}\left[\psi^+(X,Y,Z)+\psi^+(JX,JY,Z)\right]+g(B_s(X,Y),Z)
      \end{align*}
\end{remark}
\begin{corollary}
    \label{cor:HermitianOnetoOne}
    There is a one-to-one correspondence between the space of Hermitian connections on $(M,g,J,\omega)$ and $\Omega^{(1,1)}(M;TM)$.
\end{corollary}
\begin{proof}
    For $\alpha\in \Omega^{(1,1)}(M;TM)$, let $T^\alpha\in \Omega^2(M;TM)$ be given by the formula in Proposition \ref{prop:HermitianTorsionFormulaReal}.  We show that $T^\alpha=T^{\alpha'}$ implies $\alpha=\alpha'$.  From the formula in Proposition \ref{prop:HermitianTorsionFormulaReal}, we have
    \begin{align*}
        g(T^\alpha(X^+,Y^-),Z)&=-\frac{1}{4}g(N(X^+,Y^-),Z)-\frac{1}{2}(d\omega)^+(X^+,Y^-,JZ)+\frac{1}{2}(d\omega)^+(X^+,Y^-,JZ)\\
        &+\frac{1}{2}\alpha^+(X^+,Y^-,Z)-\frac{1}{2}\alpha^+(X^+,Y^-,Z)+g(\alpha(X^+,Y^-),Z)\\
        &=g(\alpha(X^+,Y^-),Z),
    \end{align*}
    where we have made use of the fact that $N(X^+,Y^-)=0$.  Hence, $T^\alpha=T^{\alpha'}$ implies $g(\alpha(X^+,Y^-),Z)=g(\alpha'(X^+,Y^-),Z)$ which in turn implies $\alpha=\alpha'$ since $\alpha,~\alpha'$ are (1,1) and $g$ is nondegenerate.  The corollary now follows from the fact that the space of Hermitian connections on $(M,g,J,\omega)$ is in one-to-one correspondence with the space $\{T^\alpha~|~\alpha\in \Omega^{(1,1)}(M;TM)\}$ which is in one-to-one correspondence with the elements of $\Omega^{(1,1)}(M;TM)$ by the above calculation.
\end{proof}
\subsection{Gauduchon connections}
\label{section:GauduchonConnection}
The Gauduchon connections is an affine line of Hermitian connections which includes both the Chern \cite{Chern1979} and Bismut (or Strominger) connections \cite{Bismut1989,Strominger1986} when $J$ is integrable.

We obtain the torsion formula for the Gauduchon connections from the formula of Proposition \ref{prop:HermitianTorsionFormulaReal} as follows.  Let 
$$
\eta^+(X,Y,Z):=(d\omega)^+(JX,JY,JZ).
$$
For $\lambda\in \mathbb{R}$, define $\alpha^\lambda\in \Omega^{(1,1)}(M;TM)$ by
$$
g(\alpha^\lambda(X,Y),Z)=\lambda\eta^+(X,Y,Z)+\lambda\eta^+(JX,JY,Z)=\lambda(d\omega)^+(JX,JY,JZ)+\lambda(d\omega)^+(X,Y,JZ)
$$
Let $F:\Omega^{(1,1)}(M;TM)\rightarrow \Omega^{3+}(M)$ be the linear map of Proposition \ref{prop:(1,1)to(2,1)+(2,1)}.  From the proof of Proposition \ref{prop:(1,1)to(2,1)+(2,1)}, we have
$$
F(\alpha^\lambda)(X,Y,Z):=g(\alpha^\lambda(X,Y),Z)+\mbox{cyclic}=4\lambda\eta^+(X,Y,Z)=4\lambda(d\omega)^+(JX,JY,JZ)
$$
Substituting $\alpha^\lambda$ into Proposition \ref{prop:HermitianTorsionFormulaReal} gives $\alpha^{\lambda+}:=F(\alpha^\lambda)$ and
\begin{align*}
        g(T^\lambda(X,Y),Z)&=-\frac{1}{4}g(N(X,Y),Z)-\frac{1}{2}(d\omega)^+(JX,JY,JZ)+\frac{1}{2}(d\omega)^+(X,Y,JZ)\\
        &+2\lambda(d\omega)^+(JX,JY,JZ)-2\lambda(d\omega)^+(X,Y,JZ)\\
        &+\lambda(d\omega)^+(JX,JY,JZ)+\lambda(d\omega)^+(X,Y,JZ)\\
        &=-\frac{1}{4}g(N(X,Y),Z)-\left(\frac{2\lambda-1}{2}\right)(d\omega)^+(X,Y,JZ)+\left(\frac{6\lambda-1}{2}\right)(d\omega)^+(JX,JY,JZ)
\end{align*}
where the last equality is the family of Gauduchon connections parameterized by $\lambda\in \mathbb{R}$.  For convenience, we can express the Gauduchon connections in a more standard form by taking $\lambda \rightarrow t/4$ and defining $\theta^c,\theta^b\in \Omega^2(M;TM)$ by
\begin{align}
    \label{eq:ChernTheta}
    g(\theta^c(X,Y),Z)&=\frac{1}{2}\left[(d\omega)^+(X,Y,JZ)-(d\omega)^+(JX,JY,JZ)\right]\\
    \label{eq:BismutTheta}
    g(\theta^b(X,Y),Z)&=(d\omega)^+(JX,JY,JZ).
\end{align}
Then 
\begin{equation}
    \label{eq:GauduchonTorsion}
    g(T^t(X,Y),Z)=-\frac{1}{4}g(N(X,Y),Z)+\left(1-\frac{t}{2}\right)g(\theta^c(X,Y),Z)+\frac{t}{2}g(\theta^b(X,Y),Z)
\end{equation}
\begin{remark}
    When $J$ is integrable, (\ref{eq:GauduchonTorsion}) reduces to
    \begin{equation}
    \label{eq:GauduchonTorsionInt}
    g(T^t(X,Y),Z)=\left(\frac{2-t}{4}\right)d\omega(X,Y,JZ)+\left(\frac{3t-2}{4}\right)d\omega(JX,JY,JZ)
\end{equation}
When $t=0$, we have
$$
g(T^{0}(X,Y),Z)=\frac{1}{2}\left[d\omega(X,Y,JZ)-d\omega(JX,JY,JZ)\right]
$$
which is the torsion formula for the Chern connection \cite{Chern1979}.  When $t=2$, we obtain the torsion formula for the Bismut or Strominger connection \cite{Bismut1989,Strominger1986}:
$$
g(T^2(X,Y),Z)=d\omega(JX,JY,JZ).
$$
The Chern connection is the unique Hermitian connection on a Hermitian manifold with the condition that the torsion has no (1,1)-part.  For the general case of almost Hermitian manifolds, the condition $T^{(1,1)}=0$ corresponds uniquely to the condition $\alpha=0$ in Proposition \ref{prop:HermitianTorsionFormulaReal}.  The Bismut connection is the unique Hermitian connection on a Hermitian manifold whose torsion tensor $g(T(X,Y),Z)$ is totally skew-symmetric.  When $J$ is not integrable, a Hermitian connection with totally skew-symmetric torsion is not always possible.  See Appendix \ref{appendix:Bismut} for details.
\end{remark}

\section{Left Invariant Hermitian Connections}
\label{sec:AlmostHermitianLieGroups}
\noindent In this section, we apply the results of Section \ref{sec:Preliminaries} to almost Hermitian manifolds of the form $(G,g,J,\omega)$, where $G$ is a Lie group (with identity element denoted as $1$) and $g$, $J$, (and hence) $\omega$ are left-invariant.  Hence, if $X,Y$ are left-invariant vector fields, $g(X,Y)$ is a constant and $JX$ is left-invariant.  We let $\mathfrak{g}:=\mbox{Lie(G)}$ denote the Lie algebra of left-invariant vector fields on $G$.  

We recall that a connection $\nabla$ on $G$ is left-invariant if $\nabla_XY\in \mathfrak{g}$ for all $X,Y\in \mathfrak{g}$.  An immediate consequence of this definition is the following:
\begin{corollary}
    \label{cor:TorsionLeftInvariant}
    Let $\nabla$ be a left-invariant connection on $G$ with torsion $T$.  Then $T\in \wedge^2\mathfrak{g}^\ast\otimes \mathfrak{g}$ where $\mathfrak{g}^\ast$ denotes the dual of $\mathfrak{g}$. (Hence, $\mathfrak{g}^\ast$ is the space of left-invariant 1-forms on $G$.)
\end{corollary}
\begin{proof}
    From the definition of $T$, we have 
    $$
        T(X,Y):=\nabla_XY-\nabla_YX-[X,Y]\in \mathfrak{g},\hspace*{0.2in}\forall~X,Y\in \mathfrak{g}.
    $$
\end{proof}
\begin{remark}
    Note that the converse to Corollary \ref{cor:TorsionLeftInvariant} is not true.  Indeed, let $h$ be any Riemannian metric on $G$ which is not left-invariant and let $Z\in \mathfrak{g}$ be nonzero.  Let $F:=h\otimes Z$ and let $\nabla$ be any left-invariant connection on $G$ with torsion $T$. By Corollary \ref{cor:TorsionLeftInvariant}, we have $T\in \wedge^2\mathfrak{g}^\ast \otimes \mathfrak{g}$.  From the definition of $F$, we have $F(X,Y)=F(Y,X)$ for all vector fields $X,Y$ on $G$.  Let $\nabla':= \nabla+F$ and let $T'$ denote the torsion of $\nabla'$.  Then $\nabla'$ is not left-invariant since for some $X,Y\in \mathfrak{g}$, one has $F(X,Y)=h(X,Y)Z\notin \mathfrak{g}$.  At the same time, the symmery of $F$ implies $T'=T\in \wedge^2\mathfrak{g}^\ast \otimes \mathfrak{g}$. 
\end{remark}
\begin{proposition}
    \label{prop:LeftInvariantConnection}
    Let $\nabla$ be a $g$-compatible connection on $(G,g,J,\omega)$ with torsion $T$.  Then $\nabla$ is left-invariant if and only if $T\in \wedge^2\mathfrak{g}^\ast \otimes \mathfrak{g}$.
\end{proposition}
\begin{proof}
    This follows from Corollary \ref{cor:TorsionLeftInvariant} and Proposition \ref{prop:TorsionUniqueness}.
\end{proof}

\noindent For convenience, we make the following definition:
\begin{definition}
    \label{def:standardbasis}
    A standard frame on $(G,g,J,\omega)$ is an orthonormal frame of left invariant vector fields
    $$
        (e_{(1)}~|~e_{(2)})=(e_1,\dots,e_n~|~e_{n+1},\dots, e_{2n}) 
    $$
    such that $Je_i=e_{n+i}$ for $i=1,\dots, n$.  
\end{definition}
\begin{remark}
The existence of a standard frame follows from an induction argument on $\dim G$.
\end{remark}
\noindent We recall that the Lie group of $n\times n$ complex matrices $GL(n,\mathbb{C})$ can be regarded as a subgroup of $GL(2n,\mathbb{R})$ via the identification
$$
A+iB\sim \left(\begin{array}{cc}
A & -B\\
B & A
\end{array}\right)
$$
\begin{proposition}
    \label{prop:standardbasis}
    Let $(e_{(1)}~|~e_{(2)})$ be any standard frame on $(G,g,J,\omega)$ and let $2n=\dim G$.  Then the set of all standard frames is given by 
    $$
        \{(e_{(1)}~|~e_{(2)})K~|~K\in U(n):=GL(n,\mathbb{C})\cap O(2n,\mathbb{R})\}.
    $$
\end{proposition}
\begin{proof}
    Express $K\in GL(2n,\mathbb{R})$ as
    $$
        K=\left(\begin{array}{cc}
            A & C\\
            B & D
        \end{array}\right)
    $$
    Then $(e_{(1)}'~|~e_{(2)}'):=(e_{(1)}~|~e_{(2)})K$ expands as 
    $$
        e_j'=\sum_i(a_{ij}e_i+b_{ij}f_i),\hspace*{0.2in} e'_{n+j}=\sum_i(c_{ij}e_i+d_{ij}f_i),\hspace*{0.1in} j=1,\dots, n
    $$
    $(e_{(1)}'~|~e_{(2)}')$ is then an orthonormal basis with respect to $g$ if and only if
    \begin{equation}
        \label{eq:standardbasis1}
        A^TA+B^TB=C^TC+D^TD=1_n,\hspace*{0.2in}A^TC+B^TD=0
    \end{equation}
    where $1_n$ denotes the $n\times n$ identity matrix.  Moreover, the condition $Je'_i=e'_{n+i}$ is simply the statement that the matrix representation of $J$ with respect to $(e_{(1)}~|~e_{(2)})$ and $(e_{(1)}'~|~e_{(2)}')$ are both 
    $$
        \left(\begin{array}{cc}
            0 & -1_n\\
            1_n & 0
        \end{array}\right)
    $$
    Since $K$ is the transition matrix between $(e_{(1)}'~|~e_{(2)}')$ and $(e_{(1)}~|~e_{(2)})$, it follows that 
    \begin{equation}
        \label{eq:standardbasis2}
        \left(\begin{array}{cc}
            A & C\\
            B & D
        \end{array}\right)^{-1}\left(\begin{array}{cc}
            0 & -1_n\\
            1_n & 0
        \end{array}\right)\left(\begin{array}{cc}
            A & C\\
            B & D
        \end{array}\right) = \left(\begin{array}{cc}
            0 & -1_n\\
            1_n & 0
        \end{array}\right)
    \end{equation}
    which implies $A=D$ and $C=-B$. Hence, $K\in GL(n,\mathbb{C})$.  This along with  (\ref{eq:standardbasis1}) implies that $K^TK=1_{2n}$.  This completes the proof.
\end{proof}
\begin{remark}
    A choice of standard frame uniquely determines the left-invariant almost Hermitian structure on $G$.  Moreover, any two standard frames which are related by a unitary matrix define the same left-invariant almost Hermitian structure on $G$.
\end{remark}
\noindent From this point forth, we fix a standard frame $(e_{(1)}~|~e_{(2)})$ on $(G,g,J,\omega)$.  We denote the structure constants of $\mathfrak{g}$ as 
\begin{align}
\label{eq:StructureConstants}
[e_i,e_j]= \sum_{k}C^k_{ij}e_k
\end{align}
For notational convenience, we set $\mg{X,Y}:=g(X,Y)$.  Since a standard frame is orthonormal, we have
\begin{align}
     \label{eq:StructureConstants1}
    C^k_{ij}=\mg{[e_i,e_j],e_k}
\end{align}
\noindent For $T\in \wedge^2\mathfrak{g}^\ast\otimes \mathfrak{g}$, we adopt the following notation for its components:
\begin{align}
    \label{notation:Torsion}
    T^k_{ij}&:=\mg{T(e_i,e_j),e_k}
\end{align}
\begin{lemma}
\label{lemma:nablaEi}
    Let $\nabla$ be a left-invariant $g$-compatible connection on $(G,g,J,\omega)$ with torsion $T$. Then
    \begin{align}
    \label{eq:nablaFormula}
        \nabla_{e_i}e_j&=\sum_k \Gamma^k_{ij}e_k
    \end{align} 
    where 
    \begin{equation}
        \label{eq:Christoffel}
        \Gamma^k_{ij}=\frac{1}{2}\left(C^k_{ij}-C^i_{jk}-C^j_{ik}+ T^k_{ij}- T^i_{jk}-T^j_{ik}\right)
    \end{equation}
\end{lemma}
\begin{proof}
Using (\ref{eq:UniqueConnectionTorsionEq}) and the left-invariance of $\mg{\cdot,\cdot}$, we have
\begin{align*}
    2\Gamma_{ij}^k&=2\mg{\nabla_{e_i}e_j,e_k}\\
    &= \mg{[e_i,e_j],e_k}-\mg{[e_j,e_k],e_i}-\mg{[e_i,e_k],e_j} + \mg{T(e_i,e_j),e_k}-\mg{T(e_j,e_k),e_i}-\mg{T(e_i,e_k),e_j}\\
    &= C^k_{ij}-C^i_{jk}-C^j_{ik}+ T^k_{ij}- T^i_{jk}-T^j_{ik}.
\end{align*}
\end{proof}

\begin{lemma}
    \label{lemma:NijenhuisLeftInvariant}
    The components $N^c_{ab}:=\mg{N(e_a,e_b),e_c}$ of the Nijenhuis tensor are given as follows: for $1\le i,j,k\le n$,
    \begin{equation}
     \label{eq:NijenhuisLeftInvariant1}
        N^k_{ij}=-C^{n+k}_{n+i,j}-C^{n+k}_{i,n+j}+C^k_{ij}-C^k_{n+i,n+j},\hspace*{0.1in} N^{n+k}_{ij}=C^k_{n+i,j}+C^k_{i,n+j}+C^{n+k}_{ij}-C^{n+k}_{n+i,n+j}
    \end{equation}
    \begin{equation}
     \label{eq:NijenhuisLeftInvariant2}
     N^{n+k}_{i,n+j}=-N^k_{ij},~N^k_{i,n+j}=N^{n+k}_{ij},~N^k_{n+i,n+j}=-N^k_{ij},~N^{n+k}_{n+i,n+j}=-N^{n+k}_{ij}
    \end{equation}
    In particular, $N=0$ if and only if $N^k_{ij}=N^{n+k}_{ij}=0$ for $1\le i,j,k\le n$.
\end{lemma}    
\begin{proof}
    Since $J(e_k)=e_{n+k}$ and $J(e_{n+k})=-e_k$, we have for $1\le i,j\le n$
    \begin{align*}
        N(e_i,e_j)&=J[Je_i,e_j]+J[e_i,Je_j]+[e_i,e_j]-[Je_i,Je_j]\\
            &=J[e_{n+i},e_j]+J[e_i,e_{n+j}]+\sum_{k=1}^n(C^k_{ij}e_k+C^{n+k}_{ij}e_{n+k})-[e_{n+i},e_{n+j}]\\
            &=J\sum_{k=1}^n(C^k_{n+i,j}e_k+C^{n+k}_{n+i,j}e_{n+k})+J\sum_{k=1}^n(C^k_{i,n+j}e_k+C^{n+k}_{i,n+j}e_{n+k})\\
            &+\sum_{k=1}^n(C^k_{ij}e_k+C^{n+k}_{ij}e_{n+k})-\sum_{k=1}^n(C^k_{n+i,n+j}e_k+C^{n+k}_{n+i,n+j}e_{n+k})\\
            &=\sum_{k=1}^n(C^k_{n+i,j}e_{n+k}-C^{n+k}_{n+i,j}e_{k})+\sum_{k=1}^n(C^k_{i,n+j}e_{n+k}-C^{n+k}_{i,n+j}e_{k})\\
            &+\sum_{k=1}^n(C^k_{ij}e_k+C^{n+k}_{ij}e_{n+k})-\sum_{k=1}^n(C^k_{n+i,n+j}e_k+C^{n+k}_{n+i,n+j}e_{n+k})\\
            &=\sum_{k=1}^n(-C^{n+k}_{n+i,j}-C^{n+k}_{i,n+j}+C^k_{ij}-C^k_{n+i,n+j})e_k+\sum_{k=1}^n(C^k_{n+i,j}+C^k_{i,n+j}+C^{n+k}_{ij}-C^{n+k}_{n+i,n+j})e_{n+k}
    \end{align*}
    which proves (\ref{eq:NijenhuisLeftInvariant1}).  (\ref{eq:NijenhuisLeftInvariant2}) follows from the identities
    $$
        N(e_i,e_{n+j})=N(e_i,Je_j)=-JN(e_i,e_j),\hspace*{0.2in}N(e_{n+i},e_{n+j})=N(Je_i,Je_j)=-N(e_i,e_j)
    $$ 
\end{proof}
\noindent From the definition of $\omega(\cdot,\cdot):=\mg{J\cdot,\cdot}$, we have
\begin{equation}
    \label{eq:omega}
    \omega(e_i,e_j)=\omega(e_{n+i},e_{n+j})=0,\hspace*{0.2in}\omega(e_i,e_{n+j})=\delta_{ij}
\end{equation}

\begin{lemma}
    \label{lemma:domega}
    For $1\le i,j,k\le n$:
    \begin{itemize}
        \item[(1)] $d\omega(e_i,e_j,e_k)=C_{ij}^{n+k}+C_{jk}^{n+i}+C_{ki}^{n+j}$
        \item[(2)] $d\omega(e_i,e_j,e_{n+k})=-C_{ij}^k+C_{j,n+k}^{n+i}+C_{n+k,i}^{n+j}$
        \item[(3)] $d\omega(e_i,e_{n+j},e_{n+k})=-C_{i,n+j}^k+C_{n+j, n+k}^{n+i}-C_{n+k, i}^j$
        \item[(4)] $d\omega(e_{n+i},e_{n+j},e_{n+k})=-C_{n+i, n+j}^k-C_{n+j, n+k}^i-C_{n+k, n+i}^j$
    \end{itemize}
\end{lemma}
\begin{proof}
    Since $\omega$ is left-invariant, it follows that 
    $$
        d\omega(X,Y,Z)=-\omega([X,Y],Z)-\omega([Y,Z],X)-\omega([Z,X],Y),\hspace*{0.1in}\forall~X,Y,Z\in \mathfrak{g}.
    $$
    Applying this to (1) gives
    \begin{align*}
        d\omega(e_i,e_j,e_k)&=-\omega([e_i,e_j],e_k)-\omega([e_j,e_k],e_i)-\omega([e_k,e_i],e_j)\\
        &=-\sum_l C^l_{ij}\omega(e_l,e_k)-\sum_l C^l_{jk}\omega(e_l,e_i)-\sum_lC^l_{ki}\omega(e_l,e_j)\\
        &=C^{n+k}_{ij}+C^{n+i}_{jk}+C^{n+j}_{ki}.
    \end{align*}
    (2)-(4) are computed the same way.
\end{proof}

\begin{lemma}
    \label{lemma:TJ}
    Let $T\in \wedge^2\mathfrak{g}^\ast\otimes \mathfrak{g}$.  Then 
    $$
        T_J(X,JY)=-JT_J(X,Y),\hspace*{0.2in} T_J(JX,JY)=-T_J(X,Y).
    $$
\end{lemma}
\begin{proof}
    We verify the first identity:
    \begin{align*}
        T_J(X,JY)&:=JT(JX,JY)+JT(X,JJY)+T(X,JY)-T(JX,JJY)\\
        &=JT(JX,JY)-JT(X,Y)+T(X,JY)+T(JX,Y)\\
        &=-J(-T(JX,JY)+T(X,Y)+JT(X,JY)+JT(JX,Y))\\
        &=-JT_J(X,Y).
    \end{align*}
    The second identity follows by applying the first identity.  
\end{proof}
\begin{proposition}
    \label{prop:HermitianTorsionIFFComponents}
    Let $\nabla$ be a left-invariant $g$-compatible connection on $(G,g,J,\omega)$ with torsion $T$.  Then $\nabla$ is Hermitian if and only if the torsion components $T^c_{ab}:=\mg{T(e_a,e_b),e_c}$ satisfy the following equations for $1\le i,j,k\le n$:
    \begin{align}
        \label{eq:HermitianTorsionIFFComponents1}
        &-T^{n+k}_{n+i,j}-T^{n+k}_{i,n+j}+T^k_{ij}-T^k_{n+i,n+j}+N^k_{ij}=0\\
        \label{eq:HermitianTorsionIFFComponents2}
        &T^{k}_{n+i,j}+T^k_{i,n+j}+T^{n+k}_{ij}-T^{n+k}_{n+i,n+j}+N^{n+k}_{ij}=0
   \end{align}
   and
   \begin{align}
        \label{eq:HermitianTorsionIFFComponents3}
        C^{n+k}_{ij}+C^{n+i}_{jk}+C^{n+j}_{ki}&=-T^{n+k}_{ij}-T^{n+i}_{jk}-T^{n+j}_{ki}\\
        \label{eq:HermitianTorsionIFFComponents4}
        -C^k_{ij}+C^{n+i}_{j,n+k}+C^{n+j}_{n+k,i}&=T^k_{ij}-T^{n+i}_{j,n+k}-T^{n+j}_{n+k,i}\\
        \label{eq:HermitianTorsionIFFComponents5}
        -C^k_{i,n+j}+C^{n+i}_{n+j,n+k}-C^j_{n+k,i}&=T^k_{i,n+j}-T^{n+i}_{n+j,n+k}+T^j_{n+k,i}\\
        \label{eq:HermitianTorsionIFFComponents6}
        -C^k_{n+i,n+j}-C^i_{n+j,n+k}-C^j_{n+k,n+i}&=T^k_{n+i,n+j}+T^i_{n+j,n+k}+T^j_{n+k,n+i}
    \end{align}
\end{proposition}
\begin{proof}
    This is an application of Proposition \ref{prop:HermitianTorsion}.  Since every vector field on $G$ can be expressed as a $C^\infty(G)$-linear combination of left-invariant vector fields, it suffices to take $X$, $Y$, and $Z$ in conditions (1) and (2) of Lemma \ref{lemma:HermitianTorsion} to be left-invariant vector fields.  From this, it follows that condition (1) of Lemma \ref{lemma:HermitianTorsion} is satisfied if and only if
    \begin{align}
        \label{eq:HermitianTorsionIFFComponents7}
        &\mg{T_J(e_i,e_j)+N(e_i,e_j),e_k}=0,\hspace*{0.2in}\mg{T_J(e_i,e_j)+N(e_i,e_j),e_{n+k}}=0\\
        \label{eq:HermitianTorsionIFFComponents8}
        &\mg{T_J(e_i,e_{n+j})+N(e_i,e_{n+j}),e_k}=0,\hspace*{0.2in}\mg{T_J(e_i,e_{n+j})+N(e_i,e_{n+j}),e_{n+k}}=0\\
        \label{eq:HermitianTorsionIFFComponents9}
        &\mg{T_J(e_{n+i},e_{n+j})+N(e_{n+i},e_{n+j}),e_k}=0,\hspace*{0.2in}\mg{T_J(e_{n+i},e_{n+j})+N(e_{n+i},e_{n+j}),e_{n+k}}=0
    \end{align}
    for all $1\le i,j,k\le n$. However, Lemma \ref{lemma:TJ} implies 
    \begin{align*}
        T_J(e_i,e_{n+j})+N(e_i,e_{n+j})&=-J\left[T_J(e_i,e_j)+N(e_i,e_j)\right]\\
        T_J(e_{n+i},e_{n+j})+N(e_{n+i},e_{n+j})&=-\left[T_J(e_i,e_j)+N(e_i,e_j)\right].
    \end{align*}
    Hence, equations (\ref{eq:HermitianTorsionIFFComponents7})-(\ref{eq:HermitianTorsionIFFComponents9}) hold if and only if (\ref{eq:HermitianTorsionIFFComponents7}) holds.  We expand the second half of (\ref{eq:HermitianTorsionIFFComponents7}) which will turn out to be (\ref{eq:HermitianTorsionIFFComponents2}).  The first part is handled similarly and corresponds to (\ref{eq:HermitianTorsionIFFComponents1}).
    \begin{align*}
        \mg{T_J(e_i,e_j)+N(e_i,e_j),e_{n+k}}&=\mg{JT(Je_i,e_j),e_{n+k}}+\mg{JT(e_i,Je_j),e_{n+k}}+\mg{T(e_i,e_j),e_{n+k}}\\
        &-\mg{T(Je_i,Je_j),e_{n+k}}+N^{n+k}_{ij}\\
        &=\mg{JT(e_{n+i},e_j),e_{n+k}}+\mg{JT(e_i,e_{n+j}),e_{n+k}}+  T_{ij}^{n+k} \\
        &-\mg{T(e_{n+i},e_{n+j}),e_{n+k}}+N^{n+k}_{ij}\\
        &=\sum_{l=1}^{2n}T^l_{n+i,j}\mg{Je_l,e_{n+k}}+\sum_{l=1}^{2n}T_{i,n+j}^l\mg{Je_l,e_{n+k}}+  T_{ij}^{n+k} -T_{n+i,n+j}^{n+k}+N^{n+k}_{ij} \\
        &=T^k_{n+i,j}+T_{i,n+j}^k+  T_{ij}^{n+k} -T_{n+i,n+j}^{n+k}+N^{n+k}_{ij} 
    \end{align*}
    Condition (2) of Lemma \ref{lemma:HermitianTorsion} holds if and only if 
    \begin{align}
        \label{eq:HermitianTorsionIFFComponents10}
        d\omega(e_i,e_j,e_k)&=\omega(T(e_i,e_j),e_k)+\omega(T(e_j,e_k),e_i)+\omega(T(e_k,e_i),e_j)\\
        \label{eq:HermitianTorsionIFFComponents11}
        d\omega(e_i,e_j,e_{n+k})&=\omega(T(e_i,e_j),e_{n+k})+\omega(T(e_j,e_{n+k}),e_i)+\omega(T(e_{n+k},e_i),e_j)\\
         \label{eq:HermitianTorsionIFFComponents12}
        d\omega(e_i,e_{n+j},e_{n+k})&=\omega(T(e_i,e_{n+j}),e_{n+k})+\omega(T(e_{n+j},e_{n+k}),e_i)+\omega(T(e_{n+k},e_i),e_{n+j})\\
         \label{eq:HermitianTorsionIFFComponents13}
        d\omega(e_{n+i},e_{n+j},e_{n+k})&=\omega(T(e_{n+i},e_{n+j}),e_{n+k})+\omega(T(e_{n+j},e_{n+k}),e_{n+i})+\omega(T(e_{n+k},e_{n+i}),e_{n+j})
    \end{align}
    for all $1\le i,j,k\le n$.  Equations (\ref{eq:HermitianTorsionIFFComponents10})-(\ref{eq:HermitianTorsionIFFComponents13}) is precisely (\ref{eq:HermitianTorsionIFFComponents3})-(\ref{eq:HermitianTorsionIFFComponents6}) respectively.  We expand (\ref{eq:HermitianTorsionIFFComponents13}) using Lemma \ref{lemma:domega}.  The others are handled similarly.
    \begin{align*}
        d\omega(e_{n+i},e_{n+j},e_{n+k})&=\omega(T(e_{n+i},e_{n+j}),e_{n+k})+\omega(T(e_{n+j},e_{n+k}),e_{n+i})+\omega(T(e_{n+k},e_{n+i}),e_{n+j})\\
        -C^k_{n+i,n+j}-C^i_{n+j,n+k}-C^j_{n+k,n+i}&=T_{n+i,n+j}^k+T_{n+j,n+k}^i+T^j_{n+k,n+i}.
    \end{align*}
\end{proof}
Let $\wedge^{(1,1)}\mathfrak{g}^\ast \subset \wedge^2\mathfrak{g}^\ast$ denote the space of real left-invariant 2-forms of type (1,1) with respect to the almost complex structure on $(G,g,J,\omega)$.  
\begin{proposition}
    \label{prop:alphaLeftInvariant}
    A Hermitian connection on $(G,g,J,\omega)$ is left-invariant if and only if $\alpha\in \Omega^{(1,1)}(G;TG)$ in Proposition \ref{prop:HermitianTorsionFormulaReal} belongs to the space $\wedge^{(1,1)}\mathfrak{g}^\ast \otimes \mathfrak{g}$.
\end{proposition}
\begin{proof}
    Let $\alpha\in \Omega^{(1,1)}(G;TG)$.  By Proposition \ref{prop:HermitianTorsionFormulaReal}, $\alpha$ determines the torsion $T$ for a Hermitian connection (which in turn uniquely determines the Hermitian connection).  By Proposition \ref{prop:LeftInvariantConnection}, the Hermitian connection with torsion $T$ is left-invariant if and only if $T\in \wedge^2\mathfrak{g}^\ast\otimes \mathfrak{g}$.  
    
      If $\alpha \in \wedge^{(1,1)}\mathfrak{g}^\ast \otimes \mathfrak{g}\subset \Omega^{(1,1)}(G;TG)$, it follows immediately from the formula in Proposition \ref{prop:HermitianTorsionFormulaReal} that $T\in \wedge^2\mathfrak{g}^\ast \otimes \mathfrak{g}$.

     Now suppose that $T\in \wedge^2\mathfrak{g}^\ast \otimes \mathfrak{g}$ and let $X$, $Y$, and $Z$ be left-invariant vector fields.  Substituting $X^+$, $Y^-$, and $Z$ into the formula for Proposition \ref{prop:HermitianTorsionFormulaReal} and simplifying gives
    $$
        \mg{T(X^+,Y^-),Z}=\mg{\alpha(X^+,Y^-),Z}.
    $$
      Expanding the above equation in terms of $X$ and $Y$ and simplifying gives
    $$
        \mg{T(X,Y)+T(JX,JY),Z}=2\mg{\alpha(X,Y),Z}
    $$
    where we have made use of the fact that $\alpha$ is (1,1).  Since the left side of the equation is a constant for all $X,Y,Z\in \mathfrak{g}$, it follows that $\alpha(X,Y)\in \mathfrak{g}$ for all $X,Y\in \mathfrak{g}$.  Hence, $\alpha \in \wedge^{(1,1)}\mathfrak{g}^\ast \otimes \mathfrak{g}$. 
\end{proof}
\begin{corollary}
    \label{cor:leftinvariantHermitianOneToOne}
    There is a one-to-one correspondence between the space of left-invariant Hermitian connections on $(G,\mg{\cdot,\cdot},J,\omega)$ and the space $\wedge^{(1,1)}\mathfrak{g}^\ast\otimes \mathfrak{g}$.
\end{corollary}
\begin{proof}
    This follows from Corollary \ref{cor:HermitianOnetoOne} and Proposition \ref{prop:alphaLeftInvariant}.
\end{proof}
\begin{lemma}
    \label{lemma:alphaComponents}
    Let $\alpha\in \wedge^2\mathfrak{g}^\ast\otimes \mathfrak{g}$.  Then $\alpha \in \wedge^{(1,1)}\mathfrak{g}^\ast\otimes \mathfrak{g}$ if and only if the components $\alpha^c_{ab}:=\mg{\alpha(e_a,e_b),e_c}$ satisfy the following conditions for all $1\le i,j\le n$ and $1\le c\le 2n$:
    \begin{itemize}
        \item[(a)] $\alpha_{ij}^c=\alpha_{n+i, n+j}^c$
        \item[(b)] $\alpha_{i,n+j}^c=-\alpha_{n+i, j}^c$
    \end{itemize}
    where all other components are determined by skew-symmetry in the lower indices.
\end{lemma}
\begin{proof}
The 2-form $\alpha$ is of type $(1,1)$ if and only if $\alpha(e_r, e_s)=\alpha(Je_r, Je_s)$ for all $1\le r, s\le 2n$.  This is equivalent to the requirement that for all $1\le i,j\le n$ and $1\le c\le 2n$ one has
\begin{align}
    \label{eq:alphaComponents1}
    \mg{\alpha(e_i,e_j),e_c}&=\mg{\alpha(e_{n+i},e_{n+j}),e_c}\\
    \label{eq:alphaComponents2}
    \mg{\alpha(e_i,e_{n+j}),e_{c}}&=-\mg{\alpha(e_{n+i},e_{j}),e_{c}}
\end{align}
Expanding (\ref{eq:alphaComponents1}) and (\ref{eq:alphaComponents2}) gives (a) and (b) respectively in Lemma \ref{lemma:alphaComponents}.
\end{proof}
\noindent For $\alpha\in \wedge^{(1,1)}\mathfrak{g}^\ast\otimes \mathfrak{g}$, we define 
$$
\alpha^{+}(X,Y,Z):=\mg{\alpha(X,Y),Z}+\mg{\alpha(Y,Z),X}+\mg{\alpha(Z,X),Y}
$$
for all vector fields $X,Y,Z$ on $G$.  Note that $\alpha^+$ is a (real) 3-form of type (2,1)+(1,2).  Following the notation of Section \ref{sec:Preliminaries}, we write $\eta^+$ to denote the (2,1)+(1,2) part of any 3-form $\eta\in \Omega^3(G)$.  
\begin{lemma}
    \label{lemma:alpha+}
    For $\alpha\in \wedge^{(1,1)}\mathfrak{g}^\ast\otimes \mathfrak{g}$ and $1\le a,b,c\le 2n$
    $$
        \alpha^+_{abc}:=\alpha^+(e_a,e_b,e_c)=\alpha^c_{ab}+\alpha^a_{bc}+\alpha_{ca}^b
    $$
    where $\alpha^c_{ab}:=\mg{\alpha(e_a,e_b),e_c}$.
\end{lemma}
\begin{proof}
    Immediate.
\end{proof}
\begin{lemma}
    \label{lemma:domega+}
    $(d\omega)^+$ is given as follows: for $1\le i,j,k\le n$
    \begin{align}
        \label{eq:domega+1}
        (d\omega)^+(e_i,e_j,e_k)&=\frac{1}{4}[3C^{n+k}_{ij}+3C^{n+i}_{jk}+3C^{n+j}_{ki}+C^i_{k,n+j}-C^{n+k}_{n+j,n+i}+C^j_{n+i,k}]\\
        \nonumber
        &+\frac{1}{4}[C^k_{j,n+i}-C^{n+j}_{n+i,n+k}+C^i_{n+k,j}-C^k_{i,n+j}+C^{n+i}_{n+j,n+k}-C^j_{n+k,i}]
    \end{align}
    \begin{align}   
        \label{eq:domega+2}
        (d\omega)^+(e_i,e_j,e_{n+k})&=\frac{1}{4}[-3C^k_{ij}+3C^{n+i}_{j,n+k}+3C^{n+j}_{n+k,i}-C^k_{n+i,n+j}-C^i_{n+j,n+k}-C^j_{n+k,n+i}]\\
        \nonumber
        &+\frac{1}{4}[C^i_{jk}-C^{n+j}_{k,n+i}-C^{n+k}_{n+i,j}-C^j_{ik}+C^{n+i}_{k,n+j}+C^{n+k}_{n+j,i}]
    \end{align}
    \begin{align}
        \label{eq:domega+3}
        (d\omega)^+(e_i,e_{n+j},e_{n+k})&=\frac{1}{4}[-3C^k_{i,n+j}+3C^{n+i}_{n+j,n+k}-3C^j_{n+k,i}-C^k_{j,n+i}+C^{n+j}_{n+i,n+k}-C^i_{n+k,j}]\\
        \nonumber
        &+\frac{1}{4}[C^j_{k,n+i}-C^{n+k}_{n+i,n+j}+C^i_{n+j,k}+C^{n+k}_{ij}+C^{n+i}_{jk}+C^{n+j}_{ki}]
    \end{align} 
    \begin{align}
        \label{eq:domega+4}
        (d\omega)^+(e_{n+i},e_{n+j},e_{n+k})&=\frac{1}{4}[-3C^k_{n+i,n+j}-3C^i_{n+j,n+k}-3C^j_{n+k,n+i}-C^k_{ij}+C^{n+i}_{j,n+k}+C^{n+j}_{n+k,i}]\\
        \nonumber
        &+\frac{1}{4}[C^j_{ik}-C^{n+i}_{k,n+j}-C^{n+k}_{n+j,i}-C^i_{jk}+C^{n+j}_{k,n+i}+C^{n+k}_{n+i,j}]
    \end{align}
\end{lemma}
\begin{proof}
    We compute only (\ref{eq:domega+1}).  The other cases are computed in the same way.  Using Proposition \ref{prop:3form(2,1)+(1,2)} and Lemma \ref{lemma:domega}, we have
    \begin{align*}
        (d\omega)^+(e_i,e_j,e_k)&=\frac{1}{4}[3d\omega(e_i,e_j,e_k)+d\omega(e_{n+i},e_{n+j},e_k)+d\omega(e_{n+i},e_j,e_{n+k})+d\omega(e_i,e_{n+j},e_{n+k})]\\
        &=\frac{1}{4}[3d\omega(e_i,e_j,e_k)-d\omega(e_k,e_{n+j},e_{n+i})-d\omega(e_j,e_{n+i},e_{n+k})+d\omega(e_i,e_{n+j},e_{n+k})]\\
        &=\frac{1}{4}[3C_{ij}^{n+k}+3C_{jk}^{n+i}+3C_{ki}^{n+j}+C_{k,n+j}^i-C_{n+j, n+i}^{n+k}+C_{n+i, k}^j]\\
        &+\frac{1}{4}[C_{j,n+i}^k-C_{n+i, n+k}^{n+j}+C_{n+k, j}^i-C_{i,n+j}^k+C_{n+j, n+k}^{n+i}-C_{n+k, i}^j]
    \end{align*}
\end{proof}
\begin{lemma}
    \label{lemma:TorsionPattern}
    Let $\alpha\in \wedge^{(1,1)}\mathfrak{g}^\ast\otimes \mathfrak{g}$ and let $\nabla^\alpha$ the left-invariant Hermitian connections whose torsion is associated to $\alpha$ by Proposition \ref{prop:HermitianTorsionFormulaReal}.  Let $T^\alpha$ denote the torsion of $\nabla^\alpha$.  Then
    \begin{equation}
        \label{eq:TorsionPattern}
        \mg{T^\alpha(X,JY),Z}=-\mg{T^\alpha(X,Y),JZ}-\frac{1}{2}\mg{N(X,Y),JZ}+\mg{\alpha(X,Y),JZ}+\mg{\alpha(X,JY),Z}
    \end{equation}
    for all vector fields $X,Y,Z$ on $G$.
\end{lemma}
\begin{proof}
    By Proposition \ref{prop:HermitianTorsionFormulaReal}, we have
    \begin{align*}
        \mg{T^\alpha(X,Y),Z}&=-\frac{1}{4}\mg{N(X,Y),Z}-\frac{1}{2}(d\omega)^+(JX,JY,JZ)+\frac{1}{2}(d\omega)^+(X,Y,JZ)\\
                &+\frac{1}{2}\alpha^+(X,Y,Z)-\frac{1}{2}\alpha^+(JX,JY,Z)+\mg{\alpha(X,Y),Z}
    \end{align*}
    It follows from Proposition \ref{prop:3form(2,1)+(1,2)} that if $\eta^+$ is a 3-form of type (2,1)+(1,2) then
    $$
        \eta^+(X,Y,Z)=\eta^+(JX,JY,Z)+\eta^+(JX,Y,JZ)+\eta^+(X,JY,JZ).
    $$
    Hence, 
    $$
        \eta^+(X,Y,Z)-\eta^+(JX,JY,Z)=\eta^+(JX,Y,JZ)+\eta^+(X,JY,JZ)
    $$
    and
    $$
        \eta^+(JX,JY,JZ)-\eta^+(X,Y,JZ)=\eta^+(X,JY,Z)+\eta^+(JX,Y,Z).
    $$
    Using these identities, we have
      \begin{align*}
        \mg{T^\alpha(X,JY),Z}&=-\frac{1}{4}\mg{N(X,JY),Z}+\frac{1}{2}(d\omega)^+(JX,Y,JZ)+\frac{1}{2}(d\omega)^+(X,JY,JZ)\\
                &+\frac{1}{2}\alpha^+(X,JY,Z)+\frac{1}{2}\alpha^+(JX,Y,Z)+\mg{\alpha(X,JY),Z}\\
                &=-\frac{1}{4}\mg{-JN(X,Y),Z}+\frac{1}{2}(d\omega)^+(X,Y,Z)-\frac{1}{2}(d\omega)^+(JX,JY,Z)\\
                &+\frac{1}{2}\alpha^+(JX,JY,JZ)-\frac{1}{2}\alpha^+(X,Y,JZ)+\mg{\alpha(X,JY),Z}\\
                &=-\frac{1}{4}\mg{N(X,Y),JZ}+\frac{1}{2}(d\omega)^+(X,Y,Z)-\frac{1}{2}(d\omega)^+(JX,JY,Z)\\
                &+\frac{1}{2}\alpha^+(JX,JY,JZ)-\frac{1}{2}\alpha^+(X,Y,JZ)+\mg{\alpha(X,JY),Z}\\
                &=-\mg{T^{\alpha}(X,Y),JZ}-\frac{1}{2}\mg{N(X,Y),JZ}+\mg{\alpha(X,Y),JZ}+\mg{\alpha(X,JY),Z}\\
    \end{align*}
    
\end{proof}
\noindent At this point, we have all the ingredients needed to compute the torsion components of a left-invariant Hermitian connection (the components of a left-invariant Hermitian connection is then computed using Lemma \ref{lemma:nablaEi}).  In particular, we also compute the torsion components of a left-invariant Gauduchon connection.  
\begin{proposition}
    \label{prop:LeftInvariantHermitianTorsion}
    Let $\alpha\in \wedge^{(1,1)}\mathfrak{g}^\ast\otimes \mathfrak{g}$ and let $\alpha^c_{ab}:=\mg{\alpha(e_a,e_b),e_c}$.  Let $\nabla^\alpha$ the left-invariant Hermitian connections whose torsion is associated to $\alpha$ by Proposition \ref{prop:HermitianTorsionFormulaReal}.  Let $T^\alpha$ denote the torsion of $\nabla^\alpha$.  Then for $1\le i,j,k\le n$
    \begin{align}
        \nonumber
        \mg{&T^\alpha(e_i,e_j),e_k}=-\frac{1}{2}C^k_{ij}-\frac{1}{4}C^j_{ik}+\frac{1}{4}C^i_{jk}+\frac{1}{2}C^k_{n+i,n+j}+\frac{1}{4}C^i_{n+j,n+k}+\frac{1}{4}C^j_{n+k,n+i}+\frac{1}{4}C^{n+i}_{j,n+k}\\
        \label{eq:LeftInvariantHermitianTorsion1}
        &+\frac{1}{4}C^{n+j}_{n+k,i}+\frac{1}{4}C^{n+i}_{k,n+j}-\frac{1}{4}C^{n+j}_{k,n+i}+\alpha^k_{ij}+\frac{1}{2}\alpha^i_{jk}+\frac{1}{2}\alpha^j_{ki}-\frac{1}{2}\alpha^{n+i}_{n+j,k}-\frac{1}{2}\alpha^{n+j}_{k,n+i}
    \end{align}
    \begin{align}
        \nonumber
        \mg{&T^\alpha(e_i,e_j),e_{n+k}}=-\frac{1}{2}C^{n+k}_{ij}+\frac{1}{2}C^{n+k}_{n+i,n+j}-\frac{1}{4}C^j_{k,n+i} -\frac{1}{4}C^i_{n+j,k}-\frac{1}{4}C^j_{i,n+k} + \frac{1}{4}C^{n+i}_{n+k,n+j} + \frac{1}{4}C^i_{j,n+k} \\
        \label{eq:LeftInvariantHermitianTorsion2}
        &-\frac{1}{4}C^{n+j}_{n+k,n+i} -\frac{1}{4}C^{n+j}_{ki} -\frac{1}{4}C^{n+i}_{jk} +\alpha^{n+k}_{ij} + \frac{1}{2}\alpha^i_{j,n+k} + \frac{1}{2}\alpha^j_{n+k,i} - \frac{1}{2}\alpha^{n+i}_{n+j,n+k} - \frac{1}{2}\alpha^{n+j}_{n+k,n+i} 
    \end{align}
    \begin{align}
        \nonumber
        \mg{&T^\alpha(e_i,e_{n+j}),e_k}=-\frac{1}{2}C^k_{n+i,j} - \frac{1}{2}C^k_{i,n+j} - \frac{1}{4}C^{n+j}_{n+i,n+k} +\frac{1}{4}C^i_{n+k,j} + \frac{1}{4}C^{n+i}_{n+j,n+k} - \frac{1}{4}C^j_{n+k,i} - \frac{1}{4}C^i_{k,n+j} \\
        \label{eq:LeftInvariantHermitianTorsion3}
        &-\frac{1}{4}C^j_{n+i,k} - \frac{1}{4}C^{n+j}_{ik} -\frac{1}{4}C^{n+i}_{kj} + \alpha^k_{i,n+j} +\frac{1}{2}\alpha^i_{n+j,k} + \frac{1}{2}\alpha^{n+j}_{ki} + \frac{1}{2}\alpha^{n+i}_{jk} + \frac{1}{2}\alpha^j_{k,n+i} 
    \end{align}
    \begin{align}
        \nonumber
        \mg{&T^\alpha(e_i,e_{n+j}),e_{n+k}}=-\frac{1}{2}C^{n+k}_{n+i,j} - \frac{1}{2}C^{n+k}_{i,n+j} + \frac{1}{4}C^i_{jk} -\frac{1}{4}C^{n+j}_{k,n+i} + \frac{1}{4}C^i_{n+j,n+k} + \frac{1}{4}C^j_{n+k,n+i} +\frac{1}{4}C^j_{ki} \\
        \label{eq:LeftInvariantHermitianTorsion4}
        &-\frac{1}{4}C^{n+i}_{n+j,k} -\frac{1}{4}C^{n+j}_{i,n+k} -\frac{1}{4}C^{n+i}_{n+k,j} +\alpha^{n+k}_{i,n+j} +\frac{1}{2}\alpha^i_{n+j,n+k} + \frac{1}{2}\alpha^{n+j}_{n+k,i} +\frac{1}{2}\alpha^{n+i}_{j,n+k} +\frac{1}{2}\alpha^j_{n+k,n+i}
    \end{align}
    \begin{align}
        \nonumber
        \mg{&T^\alpha(e_{n+i},e_{n+j}),e_k}=-\frac{1}{2} C_{n+i, n+j}^k + \frac{1}{2} C_{ij}^k -\frac{1}{4}C_{n+j, n+k}^i - \frac{1}{4}C_{n+k, n+i}^j -\frac{1}{4}C_{j, n+k}^{n+i} -\frac{1}{4}C_{n+k, i}^{n+j} - \frac{1}{4} C_{k, n+j}^{n+i}\\
        \label{eq:LeftInvariantHermitianTorsion5}
         &+\frac{1}{4} C_{ik}^{j}+\frac{1}{4}C_{kj}^i + \frac{1}{4}C_{k, n+i}^{n+j} + \alpha_{n+i, n+j}^k+\frac{1}{2}\alpha^j_{n+i,n+k}+\frac{1}{2}\alpha^{n+i}_{n+k,j}+\frac{1}{2}\alpha^{n+j}_{i,n+k}+\frac{1}{2}\alpha^i_{n+k,n+j}
    \end{align}
    \begin{align}
        \nonumber
        \mg{&T^\alpha(e_{n+i},e_{n+j}),e_{n+k}}=-\frac{1}{2}C_{n+i, n+j}^{n+k}+\frac{1}{2}C_{ij}^{n+k} +\frac{1}{4}C_{n+j, k}^i +\frac{1}{4}C_{k,n+i}^j +\frac{1}{4}C_{jk}^{n+i} +\frac{1}{4} C_{ki}^{n+j} +\frac{1}{4}C_{n+j,n+k}^{n+i}\\
        \label{eq:LeftInvariantHermitianTorsion6}
        &-\frac{1}{4}C_{n+k, i}^j-\frac{1}{4}C_{j, n+k}^i + \frac{1}{4}C_{n+k, n+i}^{n+j} + \alpha_{n+i, n+j}^{n+k} + \frac{1}{2}\alpha_{n+j, n+k}^{n+i} +\frac{1}{2}\alpha_{n+k, n+i}^{n+j} - \frac{1}{2}\alpha_{j, n+k}^i -\frac{1}{2}\alpha_{n+k, i}^j
    \end{align}
\end{proposition}
\begin{proof}
     (\ref{eq:LeftInvariantHermitianTorsion1}) is computed using Proposition \ref{prop:HermitianTorsionFormulaReal} and Lemmas \ref{lemma:NijenhuisLeftInvariant}, \ref{lemma:alphaComponents}, \ref{lemma:alpha+}, and \ref{lemma:domega+}:
     \begin{align*}
        \mg{T^\alpha(e_i,e_j),e_k}&=-\frac{1}{4}N^k_{ij}-\frac{1}{2}(d\omega)^+(e_{n+i},e_{n+j},e_{n+k})+\frac{1}{2}(d\omega)^+(e_i,e_j,e_{n+k})\\
        &+\frac{1}{2}\alpha^+(e_i,e_j,e_k)-\frac{1}{2}\alpha^+(e_{n+i},e_{n+j},e_k)+\alpha^k_{ij}\\
        &=-\frac{1}{4}[-C^{n+k}_{n+i,j}-C^{n+k}_{i,n+j}+C^k_{ij}-C^k_{n+i,n+j}]\\
        &-\frac{1}{8}[-3C^k_{n+i,n+j}-3C^i_{n+j,n+k}-3C^j_{n+k,n+i}-C^k_{ij}+C^{n+i}_{j,n+k}+C^{n+j}_{n+k,i}]\\
        &-\frac{1}{8}[C^j_{ik}-C^{n+i}_{k,n+j}-C^{n+k}_{n+j,i}-C^i_{jk}+C^{n+j}_{k,n+i}+C^{n+k}_{n+i,j}]\\
        &+\frac{1}{8}[-3C^k_{ij}+3C^{n+i}_{j,n+k}+3C^{n+j}_{n+k,i}-C^k_{n+i,n+j}-C^i_{n+j,n+k}-C^j_{n+k,n+i}]\\
        &+\frac{1}{8}[C^i_{jk}-C^{n+j}_{k,n+i}-C^{n+k}_{n+i,j}-C^j_{ik}+C^{n+i}_{k,n+j}+C^{n+k}_{n+j,i}]\\
        &+\frac{1}{2}[\alpha^k_{ij}+\alpha^i_{jk}+\alpha_{ki}^j]-\frac{1}{2}[\alpha^k_{n+i,n+j}+\alpha^{n+i}_{n+j,k}+\alpha_{k,n+i}^{n+j}]+\alpha^k_{ij}\\
        &=-\frac{1}{2}C^k_{ij}-\frac{1}{4}C^j_{ik}+\frac{1}{4}C^i_{jk}+\frac{1}{2}C^k_{n+i,n+j}+\frac{1}{4}C^i_{n+j,n+k}+\frac{1}{4}C^j_{n+k,n+i}+\frac{1}{4}C^{n+i}_{j,n+k}\\
        &+\frac{1}{4}C^{n+j}_{n+k,i}+\frac{1}{4}C^{n+i}_{k,n+j}-\frac{1}{4}C^{n+j}_{k,n+i}+\alpha^k_{ij}+\frac{1}{2}\alpha^i_{jk}+\frac{1}{2}\alpha^j_{ki}-\frac{1}{2}\alpha^{n+i}_{n+j,k}-\frac{1}{2}\alpha^{n+j}_{k,n+i}
    \end{align*}
    (\ref{eq:LeftInvariantHermitianTorsion2}) is computed the same way. (\ref{eq:LeftInvariantHermitianTorsion3}) can be obtained from Lemma \ref{lemma:TorsionPattern} via
    $$
        \mg{T^\alpha(e_i,e_{n+j}),e_k}=-\mg{T^\alpha(e_i,e_j),e_{n+k}}-\frac{1}{2}N^{n+k}_{ij}+\alpha_{ij}^{n+k}+\alpha_{i,n+j}^k
    $$
    (\ref{eq:LeftInvariantHermitianTorsion4}) is given by
    $$
        \mg{T^\alpha(e_i,e_{n+j}),e_{n+k}}=\mg{T^\alpha(e_i,e_j),e_k}+\frac{1}{2}N_{ij}^k-\alpha_{ij}^k+\alpha_{i,n+j}^{n+k}
    $$
    To compute (\ref{eq:LeftInvariantHermitianTorsion5}) and (\ref{eq:LeftInvariantHermitianTorsion6}), we apply Lemma \ref{lemma:TorsionPattern} using (\ref{eq:LeftInvariantHermitianTorsion4}) and (\ref{eq:LeftInvariantHermitianTorsion3}) respectively:
    $$
        \mg{T^\alpha(e_{n+i},e_{n+j}),e_k}=\mg{T^\alpha(e_j,e_{n+i}),e_{n+k}}+\frac{1}{2}N^{n+k}_{j,n+i}+\alpha^{n+k}_{n+i,j}+\alpha^k_{n+i,n+j}
    $$
    $$
        \mg{T^\alpha(e_{n+i},e_{n+j}),e_{n+k}}=-\mg{T^\alpha(e_j,e_{n+i}),e_k}-\frac{1}{2}N^k_{j,n+i}-\alpha^k_{n+i,j}+\alpha^{n+k}_{n+i,n+j}
    $$
\end{proof}

From Section \ref{section:GauduchonConnection}, the Gauduchon connections are parameterized by $t\in \mathbb{R}$ and the elements $\alpha^t\in\wedge^{(1,1)}\mathfrak{g}^\ast \otimes \mathfrak{g}$ which give rise to the Gauduchon connections are given by
\begin{equation}
    \label{eq:GauduchonAlphaT}
    \mg{\alpha^t(X,Y),Z}=\frac{t}{4}(d\omega)^+(JX,JY,JZ)+\frac{t}{4}(d\omega)^+(X,Y,JZ).
\end{equation}
Let $\nabla^t$ denote the Gauduchon connection associated to $t\in \mathbb{R}$; the associated torsion $T^t$ is given by
\begin{equation}
    \label{eq:t-TorsionGauduchon}
    \mg{T^t(X,Y),Z}=-\frac{1}{4}\mg{N(X,Y),Z}-\left(\frac{t-2}{4}\right)(d\omega)^+(X,Y,JZ)+\left(\frac{3t-2}{4}\right)(d\omega)^+(JX,JY,JZ)
\end{equation}
Applying Lemma \ref{lemma:domega+} and Lemma \ref{lemma:NijenhuisLeftInvariant} to (\ref{eq:t-TorsionGauduchon}) yields the components of the torsion $T^t$:
\begin{align}
    \nonumber
    &\mg{T^t(e_i,e_j),e_k}=\frac{t}{4}C^{n+k}_{n+i,j}+\frac{t}{4}C^{n+k}_{i,n+j}-\frac{1}{2}C^k_{ij}+\frac{-t+1}{2}C^k_{n+i,n+j}+\frac{1}{4}C^{n+i}_{j,n+k}+\frac{1}{4}C^{n+j}_{n+k,i}\\
    \label{eq:Tti,j,k}
    &+\frac{-2t+1}{4}C^i_{n+j,n+k}+\frac{-2t+1}{4}C^j_{n+k,n+i}+\frac{-t+1}{4}C^i_{jk}+\frac{t-1}{4}C^{n+j}_{k,n+i}+\frac{t-1}{4}C^j_{ik}+\frac{-t+1}{4}C^{n+i}_{k,n+j}
\end{align}
$~$
\begin{align}
    \nonumber
    &\mg{T^t(e_i,e_j),e_{n+k}}=\frac{-2t+1}{4}C^j_{n+i,k}+\frac{2t-1}{4}C^i_{n+j,k} + \frac{-t+1}{2}C^{n+k}_{n+i,n+j}+\frac{-t+1}{4}C^{n+j}_{n+i,n+k}+\frac{t-1}{4}C^{n+i}_{n+j,n+k}\\
    \label{eq:Tti,j,n+k}
    &-\frac{t}{4}C^k_{i,n+j} +\frac{t-1}{4}C^j_{i,n+k} + \frac{-t+1}{4}C^i_{j,n+k} -\frac{t}{4}C^k_{n+i,j}-\frac{1}{2}C^{n+k}_{ij}+\frac{1}{4}C^{n+j}_{ik}-\frac{1}{4}C^{n+i}_{jk} 
\end{align}
$~$
\begin{align}
    \nonumber
    &\mg{T^t(e_i,e_{n+j}),e_k}=\frac{-t+1}{4}C^i_{n+j,k} +\frac{t}{4}C^{n+k}_{n+i,n+j}+\frac{2t-1}{4}C^{n+j}_{n+i,n+k}-\frac{t}{4}C^{n+k}_{ij}+\frac{t-1}{4}C^{n+j}_{ik}+\frac{-t+1}{4}C^{n+i}_{jk}\\
    \label{eq:Tti,n+j,k}
    &+\frac{2t-1}{4}C^i_{j,n+k}+\frac{t-1}{2}C^k_{n+i,j}+\frac{t-1}{4}C^j_{n+i,k}+\frac{1}{4}C^{n+i}_{n+j,n+k}-\frac{1}{2}C^k_{i,n+j}+\frac{1}{4}C^j_{i,n+k}
\end{align}
$~$
\begin{align}
    \nonumber
    &\mg{T^t(e_i,e_{n+j}),e_{n+k}}=-\frac{t}{4}C^k_{n+i,n+j} + \frac{t-1}{4}C^j_{n+i,n+k} + \frac{-t+1}{4}C^i_{n+j,n+k} +\frac{t}{4}C^k_{ij} +\frac{-2t+1}{4}C^i_{jk} \\ 
    \label{eq:Tti,n+j,n+k}
    &+ \frac{t-1}{4}C^{n+j}_{i,n+k}+\frac{-t+1}{4}C^{n+i}_{j,n+k} +\frac{t-1}{2}C^{n+k}_{n+i,j}+\frac{-2t+1}{4}C^{n+j}_{n+i,k}-\frac{1}{2}C^{n+k}_{i,n+j}-\frac{1}{4}C^j_{ik}-\frac{1}{4}C^{n+i}_{n+j,k} 
\end{align}
$~$
\begin{align}
    \nonumber
    &\mg{T^t(e_{n+i},e_{n+j}),e_k}=\frac{-t+1}{4}C^{n+i}_{n+j,k} + \frac{-t+1}{2}C^k_{ij} + \frac{-t+1}{4}C^j_{ik} +\frac{t-1}{4}C^i_{jk} -\frac{t}{4}C^{n+k}_{i,n+j} + \frac{-2t+1}{4}C^{n+j}_{i,n+k}\\
    \label{eq:Ttn+i,n+j,k}
    &+\frac{2t-1}{4}C^{n+i}_{j,n+k}-\frac{t}{4}C^{n+k}_{n+i,j} +\frac{t-1}{4}C^{n+j}_{n+i,k} -\frac{1}{2}C^k_{n+i,n+j}+\frac{1}{4}C^j_{n+i,n+k}-\frac{1}{4}C^i_{n+j,n+k}
\end{align}
$~$
\begin{align}
    \nonumber
    &\mg{T^t(e_{n+i},e_{n+j}),e_{n+k}}=\frac{t-1}{4}C^{n+j}_{n+i,n+k}+\frac{-t+1}{4}C^{n+i}_{n+j,n+k} +\frac{-t+1}{2}C^{n+k}_{ij}+\frac{2t-1}{4}C^{n+j}_{ik}+\frac{t}{4}C^k_{i,n+j}\\
    \label{eq:Ttn+i,n+j,n+k}
    &+\frac{-t+1}{4}C^j_{i,n+k}+\frac{-2t+1}{4}C^{n+i}_{jk}+\frac{t-1}{4}C^i_{j,n+k}+\frac{t}{4}C^k_{n+i,j}+\frac{1}{4}C^i_{n+j,k}-\frac{1}{2}C^{n+k}_{n+i,n+j}-\frac{1}{4}C^j_{n+i,k}
\end{align}
$~$\\
\noindent We conclude this section with the following examples.
\begin{example}
    \label{ex:HermitianExample-dim4}
    Consider the Lie group $G$ with left-invariant frame $e_1,e_2,e_3,e_4$ whose non-zero bracket relations are
    $$
        [e_1,e_2]=e_2,\hspace*{0.1in}[e_1,e_3]=e_2+e_3,\hspace*{0.1in}[e_1,e_4]=e_3+e_4.
    $$
    The above 4-dimensional Lie algebra (and others) is studied in \cite{ABDO2005}. Equip $G$ with a left-invariant Hermitian structure $(J,\mg{\cdot,\cdot})$ so that $e_1,e_2,e_3,e_4$ is a standard frame with respect to $(J,\mg{\cdot,\cdot})$.  Let $\alpha\in \wedge^{(1,1)}\mathfrak{g}^\ast\otimes \mathfrak{g}$ be the left-invariant 2-form of type (1,1) with values in $\mathfrak{g}$ whose nonzero components $\alpha_{ab}^c=\mg{\alpha(e_a,e_b),e_c}$ are given by
    $$
        \alpha^1_{12}=\alpha^1_{34}=1,\hspace*{0.1in}\alpha_{23}^4=-5,\hspace*{0.1in}\alpha^4_{41}=5
    $$
    where $\alpha^c_{ab}=-\alpha^c_{ba}$.  Note that the components of $\alpha$ satisfy the conditions of Lemma \ref{lemma:alphaComponents}.  From Proposition \ref{prop:LeftInvariantHermitianTorsion}, the nonzero components of $T^\alpha$ are 
    \begin{align*}
        \mg{T^\alpha(e_1,e_2),e_1}=\frac{3}{4},\hspace*{0.1in}\mg{T^\alpha(e_1,e_4),e_1}=\frac{1}{2},\hspace*{0.1in}\mg{T^\alpha(e_2,e_3),e_1}=-\frac{1}{2},\hspace*{0.1in}\mg{T^\alpha(e_3,e_4),e_1}=\frac{5}{4}
    \end{align*}
    \begin{align*}
        \mg{T^\alpha(e_1,e_2),e_2}=\frac{3}{2},\hspace*{0.2in} \mg{T^\alpha(e_3,e_4),e_2}=-\frac{3}{2}
    \end{align*}
    \begin{align*}
        \mg{T^\alpha(e_1,e_2),e_3}=-\frac{1}{2},\hspace*{0.1in}\mg{T^\alpha(e_1,e_4),e_3}=-\frac{3}{4},\hspace*{0.1in}\mg{T^\alpha(e_2,e_3),e_3}=\frac{3}{4},\hspace*{0.1in}\mg{T^\alpha(e_3,e_4),e_3}=\frac{1}{2}
    \end{align*}
    \begin{align*}
        \mg{T^\alpha(e_1,e_4),e_4}=-3.5,\hspace*{0.2in}\mg{T^\alpha(e_2,e_3),e_4}=-6.5
    \end{align*}
    The components $\Gamma^c_{ab}$ of the associated left-invariant Hermitian connection $\nabla^\alpha$ (which are defined by $\nabla^\alpha_{e_a}e_b=\sum_c \Gamma^c_{ab}e_c$) are computed via Lemma \ref{lemma:nablaEi}:
    \begin{align*}
        &\Gamma^1_{12}=\frac{3}{4},\hspace*{0.1in} \Gamma^1_{14}=\frac{1}{2},\hspace*{0.1in}\Gamma^1_{22}=\frac{5}{2},\hspace*{0.1in}\Gamma^1_{32}=\frac{1}{2},\hspace*{0.1in}\Gamma^1_{33}=1,\hspace*{0.1in}\Gamma^1_{34}=\frac{3}{4},\hspace*{0.1in}\Gamma^1_{43}=-\frac{1}{2},\hspace*{0.1in}\Gamma^1_{44}=-\frac{5}{2}
    \end{align*}
    \begin{align*}
        \Gamma^2_{11}=-\frac{3}{4},\hspace*{0.1in}\Gamma^2_{13}=\frac{1}{2},\hspace*{0.1in}\Gamma^2_{21}=-\frac{5}{2},\hspace*{0.1in}\Gamma^2_{31}=-\frac{1}{2},\hspace*{0.1in}\Gamma^2_{33}=\frac{3}{4},\hspace*{0.1in}\Gamma^2_{34}=-4,\hspace*{0.1in}\Gamma^2_{43}=-\frac{5}{2}
    \end{align*}
    \begin{align*}
        \Gamma^3_{12}=-\frac{1}{2},\hspace*{0.1in}\Gamma^3_{14}=\frac{3}{4},\hspace*{0.1in}\Gamma^3_{24}=\frac{5}{2}\hspace*{0.1in}\Gamma^3_{31}=-1,\hspace*{0.1in}\Gamma^3_{32}=-\frac{3}{4},\hspace*{0.1in}\Gamma^3_{34}=\frac{1}{2},\hspace*{0.1in}\Gamma^3_{41}=\frac{1}{2},\hspace*{0.1in}\Gamma^3_{42}=\frac{5}{2}
    \end{align*}
    \begin{align*}
        \Gamma^4_{11}=-\frac{1}{2},\hspace*{0.1in}\Gamma^4_{13}=-\frac{3}{4},\hspace*{0.1in}\Gamma^4_{23}=-\frac{5}{2},\hspace*{0.1in}\Gamma^4_{31}=-\frac{3}{4},\hspace*{0.1in}\Gamma^4_{32}=4,\hspace*{0.1in}\Gamma^4_{33}=-\frac{1}{2},\hspace*{0.1in}\Gamma^4_{41}=\frac{5}{2}
    \end{align*}
    $\nabla^\alpha J=0$ if and only if $(\nabla^\alpha_{e_a} J)e_b=0$ for all $1\le a,b\le 4$ and the latter is easily found to be equivalent to the following for $1\le a\le 2n$ and $1\le j,k\le n$ (where $n=2$ here):
    \begin{align}
        \label{eq:NablaJ0}
        &\Gamma^k_{a,n+j}+\Gamma^{n+k}_{aj}=0,\hspace*{0.2in}\Gamma^{n+k}_{a,n+j}-\Gamma^k_{aj}=0
    \end{align}
    The condition $\nabla g=0$ is easily found to be equivalent to the following condition for all $1\le a,b,c\le 2n$:
    \begin{equation}
        \label{eq:Nablag0}
        \Gamma^c_{ab}+\Gamma^b_{ac}=0.
    \end{equation}
    One can verify by inspection that the values calculated above for $\Gamma^c_{ab}$ satisfy the above conditions. 
\end{example}
\begin{example}
    \label{ex:GauduchonExample-dim4}
    Let $G$ and $e_1,e_2,e_3,e_4$ be as in Example \ref{ex:HermitianExample-dim4} and let $(J,\mg{\cdot,\cdot})$ be the almost Hermitian structure on $G$ so that $e_1,e_2,e_3,e_4$ is a standard frame.  The nonzero torsion components $T^t$ for $t=2$ are computed from (\ref{eq:Tti,j,k})-(\ref{eq:Ttn+i,n+j,n+k}):
    $$
    \mg{T^2(e_1,e_2),e_1}=\frac{1}{4},~\mg{T^2(e_2,e_3),e_1}=1,~\mg{T^2(e_3,e_4),e_1}=-\frac{1}{4}
    $$
    $$
    \mg{T^2(e_1,e_3),e_2}=-1,~\mg{T^2(e_3,e_4),e_2}=-2
    $$
    $$
    \mg{T^2(e_1,e_2),e_3}=1,~\mg{T^2(e_1,e_4),e_3}=-\frac{1}{4},~\mg{T^2(e_2,e_3),e_3}=\frac{1}{4},~\mg{T^2(e_2,e_4),e_3}=2
    $$
    $$
        \mg{T^2(e_2,e_3),e_4}=-2
    $$
    Let $\nabla^2$ denote the Gauduchon connection with torsion $T^2$.  The components $\Gamma^c_{ab}$ of $\nabla^2$ (where $\nabla^2_{e_a}e_b=\sum_c \Gamma_{ab}^ce_c$) are given by 
    $$
        \Gamma^1_{12}=\frac{1}{4},~\Gamma^1_{22}=1,~\Gamma^1_{23}=1,~\Gamma^1_{33}=1,~\Gamma^1_{34}=\frac{1}{4},~\Gamma^1_{43}=\frac{1}{2},~\Gamma^1_{44}=1
    $$
    $$
    \Gamma^2_{11}=-\frac{1}{4},~\Gamma^2_{21}=-1,~\Gamma^2_{33}=\frac{1}{4},~\Gamma^2_{34}=-1,~\Gamma^2_{43}=1
    $$
    $$
    \Gamma^3_{14}=\frac{1}{4},~\Gamma^3_{21}=-1,~\Gamma^3_{24}=1,\Gamma^3_{31}=-1,~\Gamma^3_{32}=-\frac{1}{4},~\Gamma^3_{41}=-\frac{1}{2},~\Gamma^3_{42}=-1
    $$
    $$
    \Gamma^4_{13}=-\frac{1}{4},~\Gamma^4_{23}=-1,~\Gamma^4_{31}=-\frac{1}{4},~\Gamma^4_{32}=1,~\Gamma^4_{41}=-1
    $$
    The conditions $\nabla J=0$ and $\nabla g=0$ expressed in terms of the $\Gamma^c_{ab}$'s are given by equations (\ref{eq:NablaJ0}) and (\ref{eq:Nablag0}) respectively.  One easily verifies that the computed $\Gamma^c_{ab}$'s satisfy equations (\ref{eq:NablaJ0}) and (\ref{eq:Nablag0}).
\end{example}

\section{Curvature on Totally Real Almost Hermitian Structures}
In this section, we take $G=H\times A$ where $H$ is an arbitrary $n$-dimensional Lie group, $A$ is any $n$-dimensional abelian Lie group and $G$ is equipped with the product Lie group structure.  Hence, we have the natural identification $\mathfrak{g}\simeq \mathfrak{h}\oplus \mathfrak{a}$, where $\mathfrak{g}:=\mbox{Lie}(G)$, $\mathfrak{h}:=\mbox{Lie}(H)$, and $\mathfrak{a}:=\mbox{Lie}(A)$.  Hence, $\mathfrak{a}$ is an abelian ideal of $\mathfrak{g}$. We equip $G$ with the left-invariant almost Hermitian structure $(J,\mg{\cdot,\cdot})$ so that $G$ has a standard frame of the form 
$$
e_1,\dots, e_n,~e_{n+1},\dots, e_{2n}
$$
where $e_1,\dots, e_n$ is a basis of $\mathfrak{h}$ and $e_{n+1},\dots, e_{2n}$ is a basis of $\mathfrak{a}$.  (For notational convenience, we write $(e_i,0)\in \mathfrak{g}$ as $e_i$ and $(0,e_{n+i})\in \mathfrak{g}$ as $e_{n+i}$.) In particular, $J\mathfrak{h}=\mathfrak{a}$.  Motivated by the terminology of \cite{CPO2010,CPO2011, CO2015}, we say that $J$ is a \textit{totally real almost complex structure} with respect to $\mathfrak{h}$.  \noindent Since $\mathfrak{g}=\mathfrak{h}\oplus \mathfrak{a}$ and $\mathfrak{a}$ is abelian, we have
\begin{equation}
\label{eq:SimplifiedStructureConstants}
C^{n+k}_{ij}=C^{c}_{i,n+j}=C^c_{n+i,n+j}=0
\end{equation}
for $1\le i,j,k\le n$ and $1\le c\le 2n$.  Lemma \ref{lemma:NijenhuisLeftInvariant}, Lemma \ref{lemma:domega}, and (\ref{eq:SimplifiedStructureConstants}) imply the following:
\begin{corollary}
    \label{eq:SimplifiedNijhenuis}
    The Nijhenuis tensor associated with $J$ has components 
    $$
        N^k_{ij}=C^k_{ij},\hspace*{0.2in}N^k_{n+i,n+j}=N^{n+k}_{i,n+j}=-C^k_{ij},\hspace*{0.2in}N^k_{i,n+j}=N^{n+k}_{n+i,n+j}=N^{n+k}_{ij}=0
    $$
    for $1\le i,j,k\le n$ where $N^c_{ab}:=\mg{N(e_a,e_b),e_c}$.  In particular, $J$ is integrable if and only if $H$ is abelian.  In addition, the fundamental 2-form $\omega(\cdot,\cdot):=\mg{J\cdot,\cdot}$ is closed if and only if $H$ is abelian.
\end{corollary}

As a consequence of (\ref{eq:SimplifiedStructureConstants}), the torsion formulas (\ref{eq:LeftInvariantHermitianTorsion1})-(\ref{eq:LeftInvariantHermitianTorsion6}) and (\ref{eq:Tti,j,k})-(\ref{eq:Ttn+i,n+j,n+k}) reduce respectively to the following:
\begin{align}
        \label{eq:HA_LeftInvariantHermitianTorsion1}
        \mg{T^\alpha(e_i,e_j),e_k}&=-\frac{1}{2}C^k_{ij}-\frac{1}{4}C^j_{ik}+\frac{1}{4}C^i_{jk}+\alpha^k_{ij}+\frac{1}{2}\alpha^i_{jk}+\frac{1}{2}\alpha^j_{ki}-\frac{1}{2}\alpha^{n+i}_{n+j,k}-\frac{1}{2}\alpha^{n+j}_{k,n+i}\\
        \label{eq:HA_LeftInvariantHermitianTorsion2}
        \mg{T^\alpha(e_i,e_j),e_{n+k}}&=\alpha^{n+k}_{ij} + \frac{1}{2}\alpha^i_{j,n+k} + \frac{1}{2}\alpha^j_{n+k,i} - \frac{1}{2}\alpha^{n+i}_{n+j,n+k} - \frac{1}{2}\alpha^{n+j}_{n+k,n+i}\\ 
        \label{eq:HA_LeftInvariantHermitianTorsion3}
        \mg{T^\alpha(e_i,e_{n+j}),e_k}&= \alpha^k_{i,n+j} +\frac{1}{2}\alpha^i_{n+j,k} + \frac{1}{2}\alpha^{n+j}_{ki} + \frac{1}{2}\alpha^{n+i}_{jk} + \frac{1}{2}\alpha^j_{k,n+i}\\ 
        \label{eq:HA_LeftInvariantHermitianTorsion4}
        \mg{T^\alpha(e_i,e_{n+j}),e_{n+k}}&=\frac{1}{4}C^i_{jk}  +\frac{1}{4}C^j_{ki} +\alpha^{n+k}_{i,n+j} +\frac{1}{2}\alpha^i_{n+j,n+k} + \frac{1}{2}\alpha^{n+j}_{n+k,i} +\frac{1}{2}\alpha^{n+i}_{j,n+k} +\frac{1}{2}\alpha^j_{n+k,n+i}\\
        \label{eq:HA_LeftInvariantHermitianTorsion5}
        \mg{T^\alpha(e_{n+i},e_{n+j}),e_k}&= \frac{1}{2} C_{ij}^k +\frac{1}{4} C_{ik}^{j}+\frac{1}{4}C_{kj}^i  + \alpha_{n+i, n+j}^k+\frac{1}{2}\alpha^j_{n+i,n+k}+\frac{1}{2}\alpha^{n+i}_{n+k,j}+\frac{1}{2}\alpha^{n+j}_{i,n+k}\\
        \nonumber
        &+\frac{1}{2}\alpha^i_{n+k,n+j}\\
        \label{eq:HA_LeftInvariantHermitianTorsion6}
        \mg{T^\alpha(e_{n+i},e_{n+j}),e_{n+k}}&= \alpha_{n+i, n+j}^{n+k} + \frac{1}{2}\alpha_{n+j, n+k}^{n+i} +\frac{1}{2}\alpha_{n+k, n+i}^{n+j} - \frac{1}{2}\alpha_{j, n+k}^i -\frac{1}{2}\alpha_{n+k, i}^j
    \end{align}
    and
    \begin{align}
    \label{eq:HA_Tti,j,k}
    \mg{T^t(e_i,e_j),e_k}&=-\frac{1}{2}C^k_{ij}+\frac{-t+1}{4}C^i_{jk}+\frac{t-1}{4}C^j_{ik}\\
    \label{eq:HA_Tti,j,n+k}
    \mg{T^t(e_i,e_j),e_{n+k}}&=0\\ 
    \label{eq:HA_Tti,n+j,k}
    \mg{T^t(e_i,e_{n+j}),e_k}&=0\\
    \label{eq:HA_Tti,n+j,n+k}
    \mg{T^t(e_i,e_{n+j}),e_{n+k}}&=\frac{t}{4}C^k_{ij} +\frac{-2t+1}{4}C^i_{jk}-\frac{1}{4}C^j_{ik}\\
    \label{eq:HA_Ttn+i,n+j,k}
    \mg{T^t(e_{n+i},e_{n+j}),e_k}&=\frac{-t+1}{2}C^k_{ij} + \frac{-t+1}{4}C^j_{ik} +\frac{t-1}{4}C^i_{jk}\\
    \label{eq:HA_Ttn+i,n+j,n+k}
    \mg{T^t(e_{n+i},e_{n+j}),e_{n+k}}&=0
\end{align}
\noindent Let $\nabla^\alpha$ and $\nabla^t$ be the Hermitian and Gauduchon connection on $(G,J,\mg{\cdot,\cdot})$ whose torsion is $T^\alpha$ and $T^t$ respectively.  For completeness (and later use), we record some formulas related to the curvature of a left-invariant connection.  Recall that for an arbitrary connection $\nabla$ on $G$ the associated curvature tensor is defined by
$$
R(X,Y)Z:=[\nabla_X,\nabla_Y]Z-\nabla_{[X,Y]}Z
$$
where $X,Y,Z$ are any vector fields on $G$.  For convenience, we define
$$
R(X,Y,Z,W):=\mg{R(X,Y)Z,W}.
$$
As usual, the components of $\nabla$ are defined with respect to the standard frame $e_1,\dots, e_{2n}$ via 
$$
\Gamma^c_{ab}:=\mg{\nabla_{e_a}e_b,e_c}.
$$
In the case of a left-invariant connection one has the following:
\begin{proposition}
    \label{prop:CurvatureLeftInvariant}
    Let $\nabla$ be a left-invariant connection on $(G,J,\mg{\cdot,\cdot})$.  Then 
    \begin{equation}
        \label{eq:CurvatureLeftInvariant1}
        R(e_a,e_b,e_c,e_d)=\sum_p\left(\Gamma_{ap}^d\Gamma^p_{bc}-\Gamma_{bp}^d\Gamma_{ac}^p-C^p_{ab}\Gamma_{pc}^d\right)
    \end{equation}
\end{proposition}
\begin{proof}
    \begin{align*}
        R(e_a,e_b)e_c&=\nabla_{e_a}\nabla_{e_b}e_c-\nabla_{e_b}\nabla_{e_a}e_c-\nabla_{[e_a,e_b]}e_c\\
        &=\sum_p \left(\Gamma^p_{bc}\nabla_{e_a}e_p-\Gamma_{ac}^p\nabla_{e_b}e_p-C^p_{ab}\nabla_{e_p}e_c\right)\\
        &=\sum_q\sum_p\left(\Gamma_{ap}^q\Gamma^p_{bc}-\Gamma_{bp}^q\Gamma_{ac}^p-C^p_{ab}\Gamma_{pc}^q \right)e_q
    \end{align*}
    which implies (\ref{eq:CurvatureLeftInvariant1}).
\end{proof}
\begin{corollary}
    \label{cor:HermitianCurvature}
    Let $\nabla$ be a metric compatible left-invariant connection on $(G,J,\mg{\cdot,\cdot})$ with torsion $T$.  Then
    \begin{align}
        \nonumber
        R(e_a,e_b,e_c,e_d)&=\sum_p\left(\widehat{C}_{ap}^d\widehat{C}^p_{bc}-\widehat{C}_{bp}^d\widehat{C}_{ac}^p-C^p_{ab}\widehat{C}_{pc}^d\right)+\sum_p\left(\widehat{T}_{ap}^d\widehat{T}^p_{bc}-\widehat{T}_{bp}^d\widehat{T}_{ac}^p\right)  \\
        \label{eq:metricCurvatureLeftInvariant1}
        &+\sum_p\left(\widehat{C}_{ap}^d\widehat{T}^p_{bc}+\widehat{C}^p_{bc}\widehat{T}_{ap}^d-\widehat{C}_{bp}^d\widehat{T}_{ac}^p-\widehat{T}_{bp}^d\widehat{C}_{ac}^p-C^p_{ab}\widehat{T}_{pc}^d\right)
    \end{align}
    where $\widehat{C}^c_{ab}:=\frac{1}{2}(C^c_{ab}-C^a_{bc}-C^b_{ac})$, $\widehat{T}^c_{ab}:=\frac{1}{2}(T^c_{ab}-T^a_{bc}-T^b_{ac})$, and $T^c_{ab}:=\mg{T(e_a,e_b),e_c}$.
\end{corollary}
\begin{proof}
    From Lemma \ref{lemma:nablaEi}, $\Gamma^c_{ab}=\widehat{C}^c_{ab}+\widehat{T}^c_{ab}$.  Substituting this into (\ref{eq:CurvatureLeftInvariant1}) gives
    \begin{align*}
        R(e_a,e_b,e_c,e_d)&=\sum_p\left(\Gamma_{ap}^d\Gamma^p_{bc}-\Gamma_{bp}^d\Gamma_{ac}^p-C^p_{ab}\Gamma_{pc}^d\right)\\
        &=\sum_p(\widehat{C}_{ap}^d+\widehat{T}_{ap}^d)(\widehat{C}^p_{bc}+\widehat{T}^p_{bc})-\sum_p(\widehat{C}_{bp}^d+\widehat{T}_{bp}^d)(\widehat{C}_{ac}^p+\widehat{T}_{ac}^p)-\sum_pC^p_{ab}(\widehat{C}_{pc}^d+\widehat{T}_{pc}^d)\\
        &=\sum_p(\widehat{C}_{ap}^d\widehat{C}^p_{bc}+\widehat{C}_{ap}^d\widehat{T}^p_{bc}+\widehat{C}^p_{bc}\widehat{T}_{ap}^d+\widehat{T}_{ap}^d\widehat{T}^p_{bc})-\sum_p(\widehat{C}_{bp}^d\widehat{C}_{ac}^p+\widehat{C}_{bp}^d\widehat{T}_{ac}^p+\widehat{C}_{ac}^p\widehat{T}_{bp}^d+\widehat{T}_{bp}^d\widehat{T}_{ac}^p)\\
        &-\sum_p(C^p_{ab}\widehat{C}_{pc}^d+C^p_{ab}\widehat{T}_{pc}^d)
    \end{align*}
  which is (\ref{eq:metricCurvatureLeftInvariant1}) after rearranging the terms.    
\end{proof}
\begin{remark}
    \label{rmk:CurvatureFormulas}
    Note that the formulas in Proposition \ref{prop:CurvatureLeftInvariant} and Corollary \ref{cor:HermitianCurvature} apply to all Lie groups equipped with a left-invariant almost Hermitian structure (as opposed to only Lie groups of the form $H\times A$).
\end{remark}
\noindent For the remainder of the paper, we let
$$
R_{abcd}:=R(e_a,e_b,e_c,e_d)
$$

Let $\alpha\in \wedge^{(1,1)}\mathfrak{g}^\ast\otimes \mathfrak{g}$ and let $\nabla^\alpha$ be the Hermitian connection on $(G,\mg{\cdot,\cdot},J)$ whose torsion is $T^\alpha$.  For notational convenience, we let $T:=T^\alpha$ and $T^c_{ab}:=\mg{T(e_a,e_b),e_c}$.  The components $T^c_{ab}$ satisfy the following identities:
\begin{corollary}
    \label{cor:TabcHermitian}
    $T\in \wedge^2\mathfrak{g}^\ast\otimes \mathfrak{g}$ is the torsion of a Hermitian connection on $(G,J,\mg{\cdot,\cdot})$ if and only if for $1\le i,j,k\le n$, the components of $T$ satisfy the following identities:
    \begin{align}
        \label{eq:TabcHermitian1}
        -T^{n+k}_{n+i,j}-T^{n+k}_{i,n+j}+T^k_{ij}-T^k_{n+i,n+j}+C^k_{ij}&=0\\
        \label{eq:TabcHermitian2}
        T^{k}_{n+i,j}+T^k_{i,n+j}+T^{n+k}_{ij}-T^{n+k}_{n+i,n+j}&=0
   \end{align}
   and
   \begin{align}
        \label{eq:TabcHermitian3}
        -T^{n+k}_{ij}-T^{n+i}_{jk}-T^{n+j}_{ki}&=0\\
        \label{eq:TabcHermitian4}
        T^k_{ij}-T^{n+i}_{j,n+k}-T^{n+j}_{n+k,i}&=-C^k_{ij}\\
        \label{eq:TabcHermitian5}
        T^k_{i,n+j}-T^{n+i}_{n+j,n+k}+T^j_{n+k,i}&=0\\
        \label{eq:TabcHermitian6}
        T^k_{n+i,n+j}+T^i_{n+j,n+k}+T^j_{n+k,n+i}&=0
    \end{align}
\end{corollary}
\begin{proof}
This follows immediately from Proposition \ref{prop:HermitianTorsionIFFComponents}, Corollary \ref{eq:SimplifiedNijhenuis}, and (\ref{eq:SimplifiedStructureConstants}).
\end{proof}
\noindent The quantities $\widehat{T}^c_{ab}$ feature prominently in the formula of Corollary \ref{cor:HermitianCurvature}. Note that $\widehat{T}^c_{ab}$ (unlike $T^c_{ab}$) is not skew-symmetric in the lower indices. Instead, one easily finds that $\widehat{T}^c_{ab}$ satisfies the following identities:
\begin{equation}
    \label{eq:hatTIdent1}
    \widehat{T}^c_{ab}-\widehat{T}^c_{ba}=T^c_{ab},\hspace*{0.2in}\widehat{T}^c_{ab}-\widehat{T}^a_{cb}=-T^b_{ac},\hspace*{0.1in}\widehat{T}^c_{ab}=-\widehat{T}^b_{ac}
\end{equation}
Of course, the same identities apply to $\widehat{C}^c_{ab}$:
\begin{equation}
    \label{eq:hatCIdent1}
    \widehat{C}^c_{ab}-\widehat{C}^c_{ba}=C^c_{ab},\hspace*{0.2in}\widehat{C}^c_{ab}-\widehat{C}^a_{cb}=-C^b_{ac},\hspace*{0.1in}\widehat{C}^c_{ab}=-\widehat{C}^b_{ac}
\end{equation}
Using (\ref{eq:HA_LeftInvariantHermitianTorsion1})-(\ref{eq:HA_LeftInvariantHermitianTorsion6}), the third identity in (\ref{eq:hatTIdent1}), and the symmetries of $\alpha$, one obtains the following for $1\le i,j,k\le n$:
\begin{align}
    \label{eq:hatTijk}
    \widehat{T}_{ij}^k&=-\frac{1}{4}C^k_{ij}+\frac{1}{4}C^j_{ik}+\frac{1}{2}C^i_{jk}+\frac{1}{2}\alpha^k_{ij}-\frac{1}{2}\alpha^j_{ik}-\frac{1}{2}\alpha^{n+j}_{i,n+k}+\frac{1}{2}\alpha^{n+k}_{i,n+j}\\
    \label{eq:hatTijnk}
    \widehat{T}^{n+k}_{ij}&=\frac{1}{2}\alpha^{n+k}_{ij}-\frac{1}{2}\alpha^j_{i,n+k}+\frac{1}{2}\alpha^{n+j}_{ik}-\frac{1}{2}\alpha^k_{i,n+j}\\
     \label{eq:hatTinjk}
     \widehat{T}_{i,n+j}^k&=\frac{1}{2}\alpha_{i,n+j}^k-\frac{1}{2}\alpha_{ik}^{n+j}+\frac{1}{2}\alpha_{i,n+k}^j-\frac{1}{2}\alpha_{ij}^{n+k}\\
     \label{eq:hatTnijk}
     \widehat{T}^k_{n+i,j}&=-\frac{1}{2}\alpha^k_{i,n+j}+\frac{1}{2}\alpha_{i,n+k}^j-\frac{1}{2}\alpha_{ik}^{n+j}+\frac{1}{2}\alpha_{ij}^{n+k}\\
     \label{eq:hatTinjnk}
     \widehat{T}^{n+k}_{i,n+j}&=-\frac{1}{4}C^j_{ik}+\frac{1}{4}C^k_{ij}+\frac{1}{2}\alpha_{i,n+j}^{n+k}-\frac{1}{2}\alpha_{i,n+k}^{n+j}-\frac{1}{2}\alpha^j_{ik}+\frac{1}{2}\alpha^k_{ij}\\
     \label{eq:hatTnijnk}
     \widehat{T}^{n+k}_{n+i,j}&=-\frac{1}{4}C^j_{ik}-\frac{1}{4}C^k_{ij}-\frac{1}{2}\alpha_{i,n+j}^{n+k}-\frac{1}{2}\alpha_{ik}^j-\frac{1}{2}\alpha^{n+j}_{i,n+k}-\frac{1}{2}\alpha^k_{ij}\\
     \label{eq:hatTninjk}
     \widehat{T}^k_{n+i,n+j}&=\frac{1}{4}C^k_{ij}+\frac{1}{4}C^j_{ik}+\frac{1}{2}\alpha^k_{ij}+\frac{1}{2}\alpha^j_{ik}+\frac{1}{2}\alpha^{n+j}_{i,n+k}+\frac{1}{2}\alpha^{n+k}_{i,n+j}\\
     \label{eq:hatTninjnk}
     \widehat{T}^{n+k}_{n+i,n+j}&=-\frac{1}{2}\alpha^k_{i,n+j}+\frac{1}{2}\alpha^j_{i,n+k}-\frac{1}{2}\alpha^{n+j}_{ik}+\frac{1}{2}\alpha^{n+k}_{ij}
\end{align}
\noindent In the next result, we calculate the element $\alpha\in \wedge^{(1,1)}\mathfrak{g}^\ast\otimes \mathfrak{g}$ which corresponds to the trivial Hermitian connection, that is, $\nabla_XY= 0$ for all $X,Y\in \mathfrak{g}$ (which is clearly flat by Proposition \ref{prop:CurvatureLeftInvariant}).  We will write $\nabla =0$ to denote the trivial Hermitian connection.
\begin{proposition}
    \label{prop:HermitianFlat}
    Let $G=H\times A$ where $H$ is any $n$-dimensional Lie group and $A$ is any $n$-dimensional abelian Lie group.  Let $(J,\mg{\cdot,\cdot})$ be any left-invariant almost Hermitian structure on $G$ such that $(J,\mg{\cdot,\cdot})$ admits a standard frame of the form $e_1,\cdots e_n,~e_{n+1},\dots, e_{2n}$ where $e_1,\dots, e_n$ and $e_{n+1},\dots, e_{2n}$ correspond to frames on $H$ and $A$ respectively.  Then $\nabla^\alpha=0$ is (uniquely) given by the element $\alpha\in \wedge^{(1,1)}\mathfrak{g}^\ast \otimes \mathfrak{g}$ whose components for $1\le i,j,k,\le n$ and $1\le c\le 2n$ are
    $$
        \alpha^k_{ij}=\alpha^k_{n+i,n+j}=-\frac{1}{2}C^k_{ij},\hspace*{0.2in}\alpha^{n+k}_{ij}=\alpha^c_{i,n+j}=\alpha_{n+i,j}^c=0
    $$
    where (again) $C^c_{ab}:=\mg{[e_a,e_b],e_c}$.  
\end{proposition}
\begin{proof}
    Following the previous notation, let $\alpha^c_{ab}:=\mg{\alpha(e_a,e_b),e_c}$.   Let $T:=T^\alpha$ be the torsion associated to $\nabla^\alpha$.  We need to show that $\Gamma^c_{ab}=0$ for all $1\le a,b,c\le 2n$. From Lemma \ref{lemma:nablaEi} and the definition of $\widehat{C}^c_{ab}$ and $\widehat{T}^c_{ab}$, we have $\Gamma^c_{ab}=\widehat{C}^c_{ab}+\widehat{T}^c_{ab}$.  So we need to show  that $\widehat{T}^c_{ab}=-\widehat{C}^c_{ab}$ for the above choice of $\alpha$.  (Uniqueness of $\alpha$ follows from Corollary \ref{cor:HermitianOnetoOne}.)  From (\ref{eq:hatTijk}), we have
    \begin{align*}
        \widehat{T}^k_{ij}&=-\frac{1}{4}C^k_{ij}+\frac{1}{4}C^j_{ik}+\frac{1}{2}C^i_{jk}+\frac{1}{2}\alpha^k_{ij}-\frac{1}{2}\alpha^j_{ik}-\frac{1}{2}\alpha^{n+j}_{i,n+k}+\frac{1}{2}\alpha^{n+k}_{i,n+j}\\
        &=-\frac{1}{4}C^k_{ij}+\frac{1}{4}C^j_{ik}+\frac{1}{2}C^i_{jk}-\frac{1}{4}C^k_{ij}+\frac{1}{4}C^j_{ik}\\
        &=-\frac{1}{2}C^k_{ij}+\frac{1}{2}C^j_{ik}+\frac{1}{2}C^i_{jk}\\
        &=-\frac{1}{2}(C^k_{ij}-C^i_{jk}-C^j_{ik})\\
        &=-\widehat{C}^k_{ij}
    \end{align*}
    Since $\widehat{C}^c_{ab}=0$ whenever one of its indices is greater than $n$, it only remains to show that $\widehat{T}^c_{ab}=0$ for $a>n$, $b>n$, or $c>n$.  Equations (\ref{eq:hatTinjnk}), (\ref{eq:hatTnijnk}), and (\ref{eq:hatTninjk}) respectively give
    \begin{align*}
        \widehat{T}^{n+k}_{i,n+j}&=-\frac{1}{4}C^j_{ik}+\frac{1}{4}C^k_{ij}+\frac{1}{2}\alpha^{n+k}_{i,n+j}-\frac{1}{2}\alpha^{n+j}_{i,n+k}-\frac{1}{2}\alpha^j_{ik}+\frac{1}{2}\alpha^k_{ij}\\
        &=-\frac{1}{4}C^j_{ik}+\frac{1}{4}C^k_{ij}+\frac{1}{4}C^j_{ik}-\frac{1}{4}C^k_{ij}\\
        &=0
    \end{align*}
    \begin{align*}
        \widehat{T}^{n+k}_{n+i,j}&=-\frac{1}{4}C^j_{ik}-\frac{1}{4}C^k_{ij}-\frac{1}{2}\alpha^{n+k}_{i,n+j}-\frac{1}{2}\alpha^j_{ik}-\frac{1}{2}\alpha^{n+j}_{i,n+k}-\frac{1}{2}\alpha^k_{ij}\\
        &=-\frac{1}{4}C^j_{ik}-\frac{1}{4}C^k_{ij}+\frac{1}{4}C^j_{ik}+\frac{1}{4}C^k_{ij}\\
        &=0
    \end{align*}
`   \begin{align*}
        \widehat{T}^k_{n+i,n+j}&=\frac{1}{4}C^k_{ij}+\frac{1}{4}C^j_{ik}+\frac{1}{2}\alpha^k_{ij}+\frac{1}{2}\alpha^j_{ik}+\frac{1}{2}\alpha^{n+j}_{i,n+k}+\frac{1}{2}\alpha^{n+k}_{i,n+j}\\
        &=\frac{1}{4}C^k_{ij}+\frac{1}{4}C^j_{ik}-\frac{1}{4}C^k_{ij}-\frac{1}{4}C^j_{ik}\\
        &=0
    \end{align*}
    Lastly, from (\ref{eq:hatTijnk}), (\ref{eq:hatTinjk}), (\ref{eq:hatTnijk}), and (\ref{eq:hatTninjnk}), we immediately have
    $$
        \widehat{T}^{n+k}_{ij}=\widehat{T}^k_{i,n+j}=\widehat{T}^k_{n+i,j}=\widehat{T}^{n+k}_{n+i,n+j}=0
    $$
    This completes the proof.
\end{proof}
\begin{remark}
    The connection $\nabla^\alpha$ constructed in the proof of Proposition \ref{prop:HermitianFlat} is, in general, not a Gauduchon connection.  To see this, recall from (\ref{eq:GauduchonAlphaT}) that the $t$-Gauduchon connection is determined by $\alpha^t\in \wedge^{(1,1)}\mathfrak{g}^\ast\otimes \mathfrak{g}$ which is defined by
    $$
        \mg{\alpha^t(X,Y),Z}=\frac{t}{4}(d\omega)^+(JX,JY,JZ)+\frac{t}{4}(d\omega)^+(X,Y,JZ).
    $$
    Using (\ref{eq:domega+1})-(\ref{eq:domega+4}), we have for $1\le i,j,k\le n$
    \begin{align*}
        \mg{\alpha^t(e_i,e_j),e_k}&=\frac{t}{4}(d\omega)^+(e_{n+i},e_{n+j},e_{n+k})+\frac{t}{4}(d\omega)^+(e_i,e_j,e_{n+k})=-\frac{t}{4}C^k_{ij}\\
        \mg{\alpha^t(e_i,e_{n+j}),e_{n+k}}&=\frac{t}{4}(d\omega)^+(e_{n+i},e_j,e_k)-\frac{t}{4}(d\omega)^+(e_i,e_{n+j},e_k)=-\frac{t}{4}C^j_{ik}-\frac{t}{4}C^i_{jk}
    \end{align*}
    In order for $\alpha^t$ to agree with the $\alpha$ appearing in the proof of Proposition \ref{prop:HermitianFlat}, we require
    $$
      -\frac{t}{4}C^k_{ij}=-\frac{1}{2}C^k_{ij},\hspace*{0.2in}  -\frac{t}{4}C^j_{ik}-\frac{t}{4}C^i_{jk}=0
    $$
    For an arbitrary Lie group $H$, there is, in general, no $t\in \mathbb{R}$ which satisfies both equations simultaneously which shows that the connection appearing in the proof of Proposition \ref{prop:HermitianFlat} is a Hermitian connection which is, in general, not a $t$-Gauduchon connection.
\end{remark}

Let $\nabla^t$ be the $t$-Gauduchon connection whose curvature is $T^t$.  We now determine a sufficient condition on $H$ so that $\nabla^t$ (for some $t$) is trivial.  For notational convenience, we set $T:=T^t$ and (once again) $T^c_{ab}:=\mg{T(e_a,e_b),e_c}$.  We now calculate $\widehat{T}^c_{ab}:=\frac{1}{2}(T^c_{ab}-T^a_{bc}+T^b_{ca})$. Recall that $\widehat{C}_{ab}^c=\frac{1}{2}(C_{ab}^c-C_{bc}^a+C_{ca}^b)$. Then
\begin{align}
    \label{eq:hatTtijk}
    \widehat{T}^k_{ij}&=-\frac{1}{4}C_{ij}^{k}+\frac{2-t}{4}C_{jk}^{i}-\frac{1}{4}C_{ki}^{j}
    =-\frac{1}{2}\widehat{C}_{ij}^{k}+\frac{1-t}{4}C_{jk}^{i}\\
    \label{eq:hatTtinjnk}
    \widehat{T}^{n+k}_{i,n+j}&=\frac{1}{4}C_{ij}^{k}-\frac{t}{4}C_{jk}^i+\frac{1}{4}C_{ki}^{j}=\frac{1}{2}\widehat{C}_{ij}^{k}+\frac{1-t}{4}C_{jk}^i\\
    \label{eq:hatTtnijnk}
    \widehat{T}^{n+k}_{n+i,j}&=\frac{t-1}{4}C_{ij}^{k}+\frac{1-t}{4}C_{ki}^{j}
    =\frac{1-t}{2}\widehat{C}_{ki}^{j}+\frac{t-1}{4}C_{jk}^{i}\\
    \label{eq:hatTtninjk}
    \widehat{T}^k_{n+i,n+j}&=\frac{1-t}{4}C_{ij}^{k}+\frac{t-1}{4}C_{ki}^{j}
    =\frac{t-1}{2}\widehat{C}_{ki}^{j}+\frac{1-t}{4}C_{jk}^{i}\\
    \label{eq:hatTtzeros}
\widehat{T}^{n+k}_{ij}&=
    \widehat{T}^k_{i,n+j}=
    \widehat{T}^k_{n+i,j}=
    \widehat{T}^{n+k}_{n+i,n+j}=0.
\end{align}
In preparation for Theorem \ref{thm:GauduchonFlat}, we recall some basic facts from the theory of Lie groups (cf \cite{Hall2004, Warner1983, AB2010}):
\begin{lemma}
    \label{lemma:CompactAdInvariantMetric}
    Let $H$ be a compact Lie group.  Then $\mathfrak{h}:=\mbox{Lie}(H)$ admits a positive definite $\mbox{Ad}$-invariant metric $\eta$, that is, $\eta(\mbox{Ad}_h(X),\mbox{Ad}_h(Y))=\eta(X,Y)$ $\forall~h\in H$, $X,Y\in \mathfrak{h}$ where $\mbox{Ad}:H\rightarrow \mbox{GL}(\mathfrak{h})$ is the adjoint representation.
\end{lemma}
\begin{proof}
    Fix an orientation on $H$ and let $\Omega$ be a left-invariant volume form such that 
    \begin{equation}
        \label{eq:OmegaPositiveOrientation}
        \int_H\Omega_h=1.
    \end{equation}
    Let $\eta$ be any left-invariant (Riemannian) metric on $H$.  Define $\widehat{\eta}: \mathfrak{h}\times \mathfrak{h}\rightarrow \mathbb{R}$ be defined by
    $$
        \widehat{\eta}(X,Y):=\int_H \eta(\mbox{Ad}_{h^{-1}}(X),\mbox{Ad}_{h^{-1}}(Y))\Omega_h
    $$
    for $X,Y\in \mathfrak{h}$.  Clearly $\widehat{\eta}$ is positive definite and, since $\widehat{\eta}(X,Y)\in \mathbb{R}$ for $X,Y\in \mathfrak{h}$, $\widehat{\eta}$ also induces a left-invariant Riemannian metric on $H$.  To see that $\widehat{\eta}$ is $\mbox{Ad}$-invariant, fix $X,Y\in \mathfrak{h}$ and define $f: H\rightarrow \mathbb{R}$ by 
    $$
        f(h):=\eta(\mbox{Ad}_{h^{-1}}(X),\mbox{Ad}_{h^{-1}}(Y))\in \mathbb{R},\hspace*{0.2in}h\in H.
    $$
    Let $x\in H$ be arbitrary and let $L_x: H\rightarrow H$ be left translation by $x$.  Then for $h\in H$
    \begin{align*}
        (L_x^\ast f)(h)&=f(xh)\\
        &=\eta(\mbox{Ad}_{(xh)^{-1}}(X),\mbox{Ad}_{(xh)^{-1}}(Y))\\
        &=\eta(\mbox{Ad}_{h^{-1}x^{-1}}(X),\mbox{Ad}_{(h^{-1}x^{-1}}(Y))\\
        &=\eta(\mbox{Ad}_{h^{-1}}\circ \mbox{Ad}_{x^{-1}}(X),\mbox{Ad}_{h^{-1}}\circ \mbox{Ad}_{x^{-1}}(Y))
    \end{align*}
    From this, we have  
    \begin{align*}
        \widehat{\eta}(\mbox{Ad}_{x^{-1}}X,\mbox{Ad}_{x^{-1}}(Y))&=\int_H\eta(\mbox{Ad}_{h^{-1}}\circ \mbox{Ad}_{x^{-1}}(X),\mbox{Ad}_{h^{-1}}\circ \mbox{Ad}_{x^{-1}}(Y))\Omega_h\\
        &=\int_H (L_x^\ast f)(h)\Omega_h\\
        &=\int_H(L_x^\ast f)(h)(L_x^\ast \Omega)_h\\
        &=\int_HL_x^\ast( f\Omega)_h\\
        &=\int_H f(h)\Omega_h\\
        &=\int_H\eta(\mbox{Ad}_{h^{-1}}(X),\mbox{Ad}_{h^{-1}}(Y))\Omega_h\\
        &=\widehat{\eta}(X,Y)
    \end{align*}
    where we have used the fact that $\Omega$ is left-invariant in the third equality and, in the fifth equality, we have used the fact that $L_x:H\rightarrow H$ is an orientation preserving diffeomorphism (since $L_x^\ast \Omega=\Omega$ and $\Omega$ is positively oriented by (\ref{eq:OmegaPositiveOrientation})).
\end{proof}
\begin{corollary}
    \label{cor:CsymmetryCompact}
    Let $H$ be a compact Lie group and let $\widehat{\eta}$ be a positive-definite $\mbox{Ad}$-invariant metric on $\mathfrak{h}$.  If $e_1,\dots, e_n$ is an orthonormal frame with respect to $\widehat{\eta}$ and $C^k_{ij}:=\eta([e_i,e_j],e_k)$, then $C^k_{ij}=-C^j_{ik}$.
\end{corollary}
\begin{proof}
    Let $X\in \mathfrak{h}$ and let $h(t)=\mbox{exp}(tX)$.  Then for $Y,Z\in \mathfrak{h}$, $\mbox{Ad}$-invariance of $\widehat{\eta}$ gives
    \begin{align*}
        \widehat{\eta}(\mbox{Ad}_{h(t)}Y,\mbox{Ad}_{h(t)}Z)=\widehat{\eta}(Y,Z).
    \end{align*}
    Differentiating both sides at $t=0$ gives
    $$
        \widehat{\eta}(\mbox{ad}_XY,Z)+\widehat{\eta}(Y,\mbox{ad}_XZ)=0.
    $$
    Hence,
    $$
        \widehat{\eta}([X,Y],Z)=-\widehat{\eta}(Y,[X,Z]).
    $$
    Taking $X=e_i$, $Y=e_j$, and $Z=e_k$ gives
    $$
        C_{ij}^k=-C_{ik}^j
    $$
\end{proof}
\begin{remark}
    \label{rmk:SemisimpleKillingForm}
    The proof of Lemma \ref{lemma:CompactAdInvariantMetric} provides a means of constructing a positive definite \mbox{Ad}-invariant metric on the Lie algebra $\mathfrak{h}$ of a compact Lie group $H$.  However, as the construction requires an integration, it would be nice to have an alternate way of obtaining a positive definite $\mbox{Ad}$-invariant metric on the Lie algebra. For the case when $H$ is both compact and semisimple, one can obtain such a metric by turning to the Killing form.  Specifically, if $H$ is both compact and semisimple, then the Killing form $\mathcal{K}$ which we recall is defined by
    $$
        \mathcal{K}(X,Y)=\mbox{Tr}(\mbox{ad}_X\circ \mbox{ad}_Y)
    $$
    is a negative definite $\mbox{Ad}$-invariant metric on $\mathfrak{h}$ where $\mbox{Tr}$ denotes the trace (see Theorem 2.28 of \cite{AB2010}).  Hence, if $\lambda<0$ is any negative number, then $\widehat{\eta}:=\lambda \mathcal{K}$ is a positive definite $\mbox{Ad}$-invariant metric on $\mathfrak{h}$.
\end{remark}
\noindent We now come to the main result of this section.
\begin{theorem}
    \label{thm:GauduchonFlat}
    Let $H$ be an $n$-dimensional compact Lie group and let $A$ be any $n$-dimensional abelian Lie group.  Then $G=H\times A$ admits a left-invariant almost Hermitian structure $(J,\mg{\cdot,\cdot})$ such that 
    \begin{itemize}
        \item[(1)] $J$ is totally real with respect to $\mathfrak{h}:=\mbox{Lie}(H)$, and
        \item[(2)] $\nabla^t=0$ for $t=2$ and has totally skew-symmetric torsion 
        \item[(3)] $\nabla^t$ is flat (and nontrivial) for $t=-2$
    \end{itemize}
    Moreover, if $T^B$ denotes the torsion of the Gauduchon connection $\nabla^2$ and $\beta\in \wedge^3\mathfrak{g}^\ast$ is the left-invariant 3-form defined by $\beta(X,Y,Z):=\mg{T^B(X,Y),Z}$, then $\beta$ is closed.
\end{theorem}
\begin{proof}
    Since $H$ is compact, there exists a positive definite $\mbox{Ad}$-invariant metric $\widehat{\eta}$ on $\mathfrak{h}$ by Lemma \ref{lemma:CompactAdInvariantMetric}. Let $e_1,\dots, e_n$ be any left-invariant frame on $H$ which is orthonormal with respect to $\widehat{\eta}$ and let $C^k_{ij}:=\widehat{\eta}([e_i,e_j],e_k)$.  By Corollary \ref{cor:CsymmetryCompact}, the structure constants satisfy the symmetry condition $C^k_{ij}=-C^j_{ik}$.  Let $e_{n+1},\dots, e_{2n}$ be any left-invariant frame on $A$.  Define $(J,\mg{\cdot,\cdot})$ to be the left-invariant almost Hermitian structure on $G=H\times A$ so that $e_1,\dots,e_{2n}$ is a standard frame for $(J,\mg{\cdot,\cdot})$. (Note that from the definition of standard frame, we see that $\mg{\cdot,\cdot}$ restricted to $e_1,\dots, e_n$ is precisely $\widehat{\eta}$.) From the symmetry condition on $C^k_{ij}$ for $1\le i,j,k,\le n$, we immediately have
    $$
        \widehat{C}^k_{ij}=\frac{1}{2}(C^k_{ij}-C^i_{jk}-C^j_{ik})=\frac{1}{2}C^k_{ij}.
    $$
    From (\ref{eq:hatTtijk})-(\ref{eq:hatTtzeros}), all of the $\widehat{T}^c_{ab}$ components associated to $\nabla^t$ vanish except possibly $\widehat{T}^k_{ij}$ and $\widehat{T}^{n+k}_{i,n+j}$ which reduce to 
    $$
        \widehat{T}^k_{ij}=-\frac{t}{4}C^k_{ij},\hspace*{0.2in}\widehat{T}^{n+k}_{i,n+j}=\frac{2-t}{4}C^k_{ij}
    $$
    Setting $t=2$, we obtain $\widehat{T}^c_{ab}=0$ whenever one of its indices is greater than $n$ and $\widehat{T}^k_{ij}=-\frac{1}{2}C^k_{ij}=-\widehat{C}^k_{ij}$.  From this, it follows that $\Gamma^c_{ab}:=\mg{\nabla^2_{e_a}e_b,e_c}=0$.  This implies that the torsion $T^2$ of $\nabla^2$ is given by
    \begin{equation}
    	\label{eq:SKT1}
        \mg{T^2(e_a,e_b),e_c}=-\mg{[e_a,e_b],e_c}=-C^c_{ab}
    \end{equation}
    Since $C^k_{ij}=-C^j_{ik}$ for $1\le i,j,k\le n$ and $C^c_{ab}=0$ whenever $a$, $b$, or $c$ exceeds $n$, it follows that $\mg{T^2(X,Y),Z}$ is totally skew-symmetric.  This completes the proof of (1) and (2).

    For (3), consider $\nabla^t$ for arbitrary $t$.  Given the values of $\widehat{C}^c_{ab}$ and $\widehat{T}^c_{ab}$ above, it follows from Corollary \ref{cor:HermitianCurvature} that all the components of $R^t$ are zero except possibly $R^t_{ijkl}$ and $R^t_{ij,n+k,n+l}$ (where $1\le i,j,k,l\le n$) which reduce to 
    \begin{align*}
        R^t_{ijkl}&=\frac{1}{4}\sum_{p=1}^n\left(C^l_{ip}C^p_{jk}-C^l_{jp}C^p_{ik}-2C^p_{ij}C^l_{pk}\right)+\frac{t^2}{16}\sum_{p=1}^n\left(C^l_{ip}C^p_{jk}-C^l_{jp}C^p_{ik}\right)\\
        &-\frac{t}{8}\sum_{p=1}^n\left(C^l_{ip}C^p_{jk}+C^p_{jk}C^l_{ip}-C^l_{jp}C^p_{ik}-C^l_{jp}C^p_{ik}-2C^p_{ij}C^l_{pk}\right)
    \end{align*}
    \begin{align*}
        R^t_{ij,n+k,n+l}&=\frac{2-t}{4}\sum_{p=1}^n\left(\frac{2-t}{4}C^l_{ip}C^p_{jk}-\frac{2-t}{4}C^l_{jp}C^p_{ik}-C^p_{ij}C^l_{pk}  \right)
    \end{align*}
    The Jacobi identity expressed in terms of the structure constants is 
    $$
        \sum_{p=1}^n\left(C^p_{ij}C_{pk}^l+C^p_{jk}C^l_{pi}+C^p_{ki}C^l_{pj}\right)=0
    $$
    Applying the Jacobi identity to $R^t_{ijkl}$ and $R^t_{ij,n+k,n+l}$ gives
    $$
        R^t_{ijkl}=R^t_{ij,n+k,n+l}=\left(\frac{t^2}{16}-\frac{1}{4}\right)\sum_{p=1}^nC^p_{ij}C^l_{pk}
    $$
    From this, we see that $R^t_{ijkl}=R^t_{ij,n+k,n+l}=0$ for $t=2$ (which we recognize as the trivial left-invariant connection) and $t=-2$.  This completes the proof of (3).
    
    For the last statement, let $T^B:=T^2$ (where we use ``$B$" to emphasize that the $t=2$-Gauduchon connection corresponds to the Bismut (or Strominger) connection in the integrable case).  From (\ref{eq:SKT1}), we have 
    $$
    	\beta(X,Y,Z):=\mg{T^B(X,Y),Z}=-\mg{[X,Y],Z}\hspace*{0.2in}\forall~X,Y,Z\in \mathfrak{g},
    $$
    where the above calculation shows that $\beta$ is a (left-invariant) 3-form.  For $W,X,Y,Z\in \mathfrak{g}$, we have
    \begin{align}
    	\nonumber
    	d\beta(W,X,Y,Z)&=-\beta([W,X],Y,Z)+\beta([W,Y],X,Z)-\beta([W,Z],X,Y)\\
	\nonumber
	&-\beta([X,Y],W,Z)+\beta([X,Z],W,Y)-\beta([Y,Z],W,X)\\
	\nonumber
	&=\mg{[[W,X],Y],Z}-\mg{[[W,Y],X],Z}+\mg{[[W,Z],X],Y}\\
	\nonumber
	&+\mg{[[X,Y],W],Z}-\mg{[[X,Z],W],Y}+\mg{[[Y,Z],W],X}\\
	\nonumber
	&=\mg{[[W,X],Y]+[[X,Y],W] +[[Y,W],X] ,Z}\\
	\nonumber
	&+\mg{[[W,Z],X]+[[Z,X],W],Y}+\mg{[[Y,Z],W],X}\\
	\label{eq:SKT2}
	&=\mg{[[W,X],Z],Y}+\mg{[[Y,Z],W],X}
    \end{align}
    where we have applied the Jacobi identity in the last equality.   It follows from (\ref{eq:SKT2}) that $d\beta(e_a,e_b,e_c,e_d)=0$ whenever any of the indices $a$, $b$, $c$, or $d$ exceeds $n$.  For $1\le i,j,k,l\le n$, we have
    \begin{align}
    \nonumber
    d\beta(e_i,e_j,e_k,e_l)&=\mg{[[e_i,e_j],e_l],e_k}+\mg{[[e_k,e_l],e_i],e_j}\\
    \nonumber
    &=\sum_{p=1}^n\left(C_{ij}^pC_{pl}^k+C_{kl}^pC_{pi}^j\right)\\
    \nonumber
    &=\sum_{p=1}^n\left(C_{ij}^pC_{pl}^k+C_{lk}^pC_{ip}^j\right)\\
    \nonumber
    &=\sum_{p=1}^n\left(C_{ij}^pC_{pl}^k+C_{lp}^kC_{ij}^p\right)\\
    &=0
    \end{align}
    where we have used the symmetry condition $C^k_{ij}=-C^j_{ik}$ in the fourth equality.  Hence, $d\beta=0$.
\end{proof}
\begin{remark}
\label{remark:HA}
It is interesting to note that the $t=2$-Gauduchon connection is the trivial left-invariant connection under the conditions of Theorem \ref{thm:GauduchonFlat} and has totally skew-symmetric torsion (which is not guaranteed for the non-integrable case-see Appendix \ref{appendix:Bismut}).  Recall from Section \ref{section:GauduchonConnection} that the $2$-Gauduchon connection corresponds to the Strominger or Bismut connection in the integrable case.  

With $T^B$ denoting the torsion of the $2$-Gauduchon connection, Theorem \ref{thm:GauduchonFlat} shows that the associated (left-invariant) 3-form $\beta(X,Y,Z):=\mg{T^B(X,Y),Z}$ is closed.  This is precisely the \textit{strong K\"{a}hler with torsion} (or SKT) condition which was first introduced in the context of physics  (cf \cite{Strominger1986,Gates1984}) and then became a topic of interest in complex geometry (see \cite{Fino2009} and the references therein).   In general, the 3-form $\beta$ induced by $T^B$ in Theorem \ref{thm:GauduchonFlat} is \textit{not} exact (see Example \ref{example:HA} below).  

The SKT condition has mainly been studied for the case where the almost complex structure is integrable.  In fact, the standard definition of an SKT manifold assumes a Hermitian manifold rather than an almost Hermitian one.  The reason, of course, is due to the fact that the existence of a Hermitian connection with totally skew-symmetric torsion is always guaranteed for a Hermitian manifold (and this connection is precisely the well-known Bismut/Strominger connection). So interestingly, Theorem \ref{thm:GauduchonFlat} yields a left-invariant almost Hermitian manifold $(G=H\times A,\mg{\cdot,\cdot},J)$ which satisfies the SKT condition despite the fact that $J$ is \textit{non-integrable} (when $H$ is nonabelian).  Following the terminology of the literature, $\mg{\cdot,\cdot}$ is an SKT-metric for $(G=H\times A,J)$.
\end{remark} 
\noindent We conclude the paper with the following example:
\begin{example}
    \label{example:HA}
    Let $G=SO(3)\times \mathbb{R}^3$ where $SO(3)$ is the Lie group of $3\times 3$ orthogonal (real) matrices of determinant $1$ and $\mathbb{R}^3$ is given its natural abelian Lie group structure.  Let $(J,\mg{\cdot,\cdot})$ be the left-invariant almost Hermitian structure on $G$ with standard frame $e_1,\dots, e_6$ where $e_4,e_5,e_6$ is any left-invariant frame on $\mathbb{R}^3$ and $e_1,e_2,e_3$ is the left invariant frame on $SO(3)$ whose bracket relations are 
    $$
        [e_1,e_2]=-e_3,\hspace*{0.1in}[e_1,e_3]=e_2,\hspace*{0.1in}[e_2,e_3]=-e_1.
    $$
    Note that the frame on $SO(3)$ has been chosen so that the symmetry condition $C^k_{ij}=-C^j_{ik}$ is satisfied for $1\le i,j,k\le 3$.  From the proof of Theorem \ref{thm:GauduchonFlat}, the $2$-Gauduchon connection $\nabla^2$ associated to $(SO(3)\times \mathbb{R}^3,J,\mg{\cdot,\cdot})$ is the trivial left-invariant connection and has totally skew-symmetric torsion.  In addition, the left-invariant $3$-form $\beta$ defined by
    $$
   	\beta(X,Y,Z):=\mg{T^B(X,Y),Z}=-\mg{[X,Y],Z},\hspace*{0.2in}\forall~X,Y,Z\in \mathfrak{g}
    $$ 
    is closed, where $T^B$ (again) denotes the torsion of $\nabla^2$.  Explicitly, we see that $\beta(e_1,e_2,e_3)=1$ and $\beta(e_a,e_b,e_c)=0$ whenever $a$, $b$, or $c$ exceeds $3$.  We now show that $\beta$ is not exact.  Let $\pi: SO(3)\times \mathbb{R}^3\rightarrow SO(3)$ be the natural projection map.  Since $\mathbb{R}^3$ is contractible, it follows that $\pi$ is a homotopy equivalence.  Hence, $\pi$ induces an isomorphism between the de Rham cohomology groups of $SO(3)\times \mathbb{R}^3$ and $SO(3)$.  In particular, we have
    $$
    	\pi^\ast: H^3(SO(3))\stackrel{\sim}{\longrightarrow} H^3(SO(3)\times \mathbb{R}^3).
    $$
    Since $SO(3)$ is compact, connected, and orientable of dimension $3$, we have $H^3(SO(3))\simeq \mathbb{R}$.  Any volume form on $SO(3)$ is then a generator of $H^3(SO(3))$.  Let $\mu$ be the left-invariant volume form on $SO(3)$ defined by $\mu(e_1,e_2,e_3)=1$.  Then $\beta=\pi^\ast \mu$ is a generator for $H^3(SO(3)\times \mathbb{R}^3)$ which shows that $\beta$ is not exact.
    
 The proof of Theorem \ref{thm:GauduchonFlat} also shows that $\nabla^{-2}$ is nontrivially flat.  Calculating the components $\Gamma^c_{ab}:=\mg{\nabla^{-2}_{e_a}e_b,e_c}$ we obtain the following:
    $$
        \Gamma^1_{23}=-1,~\Gamma^1_{32}=1,~\Gamma^2_{13}=1,~\Gamma^2_{31}=-1,~\Gamma^3_{12}=-1,~\Gamma^3_{21}=1
    $$
    $$
        \Gamma^4_{26}=-1,~\Gamma^4_{35}=1,~\Gamma^5_{16}=1,~\Gamma^5_{34}=-1,~\Gamma^6_{15}=-1,~\Gamma^6_{24}=1
    $$
    The cautious reader can verify directly that the connection defined by the above components is both Hermitian and flat.
\end{example}

\appendix
\section{}
\label{appendix:Bismut}
\noindent The following result was originally proved as part of Theorem 10.1 in \cite{FI2002}.  We give a proof\footnote{The authors were originally unaware of the existence of Theorem 10.1 in \cite{FI2002} at the time Proposition \ref{prop:BismutIntegrable} was proved.  The authors are grateful to Prof. Stefan Ivanov for sharing the reference to Theorem 10.1 in \cite{FI2002}.} of this interesting fact for the sake of completeness. 

\begin{proposition}
    \label{prop:BismutIntegrable}
    Let $(M,g,J,\omega)$ be an almost Hermitian manifold. Then  $(M,g,J,\omega)$ admits a Hermitian connection whose torsion tensor is totally skew-symmetric if and only if $g(N(X,Y),Z)$ is totally skew-symmetric. Moreover, if  $g(N(X,Y),Z)$ is totally skew-symmetric, then the Hermitian connection with totally skew-symmetric torsion is precisely the $t=2$-Gauduchon connection (see (\ref{eq:GauduchonTorsion})).  In this case, the torsion $T^2$ of the $t=2$-Gauduchon connection can be simplified to
    $$
        g(T^2(X,Y),Z)=d\omega(JX,JY,JZ)-g(N(X,Y),Z).
    $$
\end{proposition}
\begin{proof}
    Suppose $\nabla$ is a Hermitian connection such that $g(T(X,Y),Z)$ is skew-symmetric in $X$, $Y$, and $Z$.  By expanding $T_J$ in condition (1) of Lemma \ref{lemma:HermitianTorsion}, condition (1) can be rewritten as
    \begin{equation}
        \label{eq:BismutIntegrable1}
        g(T(JX,JY),Z)=g(N(X,Y),Z)+g(T(X,Y),Z)-g(T(JX,Y),JZ)-g(T(X,JY),JZ)
    \end{equation}
    Condition (2) of Lemma \ref{lemma:HermitianTorsion} can be rewritten as
    \begin{equation}
        \label{eq:BismutIntegrable2}
        d\omega(JX,JY,JZ)=g(T(JX,JY),Z)+g(T(JY,JZ),X)+g(T(JZ,JX),Y)
    \end{equation}
    Substituting (\ref{eq:BismutIntegrable1}) into (\ref{eq:BismutIntegrable2}) and then using the assumption that $g(T(\cdot,\cdot),\cdot)$ is skew-symmetric in its arguments, we find that
    \begin{equation}
        \label{eq:BismutIntegrable4}
        g(T(X,Y),Z)=d\omega(JX,JY,JZ)-g(N(X,Y),Z)
    \end{equation}
    This shows that if $T$ is totally skew-symmetric, then it must be unique.  Since the left side is totally skew-symmetric and $d\omega(J\cdot,J\cdot,J\cdot)$ is totally skew-symmetric, it follows that $g(N(X,Y),Z)$ is totally skew-symmetric.  We now show that (\ref{eq:BismutIntegrable4}) is precisely the torsion $T^2$ of the $t=2$ Gauduchon connection (\ref{eq:GauduchonTorsion}).  To do this, we first expand $g(T_J(X,Y),Z)$ with $T$ defined by (\ref{eq:BismutIntegrable4}).  This gives
    \begin{equation}
        \label{eq:BismutIntegrable5}
        g(T_J(X,Y),Z)=-d\omega(X,JY,Z)-d\omega(JX,Y,Z)+d\omega(JX,JY,JZ)-d\omega(X,Y,JZ)-4g(N(X,Y),Z)
    \end{equation}
    At the same time, condition (1) of Lemma \ref{lemma:HermitianTorsion} gives $g(T_J(X,Y),Z)=-g(N(X,Y),Z)$.  Substituting this into (\ref{eq:BismutIntegrable5}) gives
    \begin{equation}
        \label{eq:BismutIntegrable6}
        3g(N(X,Y),Z)=-d\omega(X,JY,Z)-d\omega(JX,Y,Z)+d\omega(JX,JY,JZ)-d\omega(X,Y,JZ)
    \end{equation}
    Applying (\ref{eq:BismutIntegrable6}) to (\ref{eq:BismutIntegrable4}), we obtain the following:
    \begin{align*}
        g(T(X,Y),Z)&=d\omega(JX,JY,JZ)-g(N(X,Y),Z)\\
        &=d\omega(JX,JY,JZ)-\frac{1}{4}\left(g(N(X,Y),Z)+3g(N(X,Y),Z)\right)\\
        &=-\frac{1}{4}g(N(X,Y),Z)+d\omega(JX,JY,JZ)-\frac{1}{4}3g(N(X,Y),Z)\\
        &=-\frac{1}{4}g(N(X,Y),Z)+d\omega(JX,JY,JZ)\\
        &-\frac{1}{4}\left(-d\omega(X,JY,Z)-d\omega(JX,Y,Z)+d\omega(JX,JY,JZ)-d\omega(X,Y,JZ)\right)\\
        &=-\frac{1}{4}g(N(X,Y),Z)+\frac{1}{4}\left(3d\omega(JX,JY,JZ)+d\omega(X,Y,JZ)+d\omega(X,JY,Z)+d\omega(JX,Y,Z) \right)\\
        &=-\frac{1}{4}g(N(X,Y),Z)+(d\omega)^+(JX,JY,JZ)\\
        &=g(T^2(X,Y),Z)
    \end{align*}
    where we applied (\ref{eq:BismutIntegrable6}) in the fourth equality, the second to last equality follows from Proposition \ref{prop:3form(2,1)+(1,2)}, and the last equality is $T^t$ in (\ref{eq:GauduchonTorsion}) for $t=2$.

    On the other hand, if $g(N(X,Y),Z)$ is totally skew-symmetric, then it follows that $g(T^2(X,Y),Z)$ is totally skew-symmetric.  The previous calculation shows that if a Hermitian connection has totally skew-symmetric torsion, then the Hermitian connection is unique and must be $T^2$.
\end{proof}



\end{document}